\documentclass[10pt]{amsart}

\usepackage[dvipsnames]{xcolor}
\usepackage[hyperfootnotes=false]{hyperref}
\usepackage{enumerate}
\usepackage{tikz}
\usepackage{tikz-cd}
\usepackage{hyperref}
\usepackage{amsthm,amsfonts,amsmath,amscd,amssymb}
\usepackage{xcolor,float}
\usepackage[all]{xy}
\usepackage{mathrsfs}
\usepackage{graphics,graphicx}
\usepackage{caption}
\usepackage{comment}

\let\oldmarginpar\marginpar
\renewcommand\marginpar[1]{\-\oldmarginpar[\raggedleft\footnotesize #1]%
	{\raggedright\footnotesize #1}}
\theoremstyle{plain}
\newtheorem{thm}{Theorem}[section]
\newtheorem{lemma}[thm]{Lemma}
\newtheorem{example}[thm]{Example}
\newtheorem{prop}[thm]{Proposition}

\newtheorem{cor}[thm]{Corollary}

\theoremstyle{definition}

\newtheorem{definition}[thm]{Definition}

\newtheorem{remark}[thm]{Remark}

\numberwithin{equation}{section}

\renewcommand{\S}{\mathbb{S}}
\renewcommand{\P}{\mathbb{P}}

\newcommand{\D}{\mathbb{D}}
\newcommand{\N}{\mathbb{N}}
\newcommand{\Z}{\mathbb{Z}}

\newcommand{\R}{\mathbb{R}}

\newcommand{\SA}{\mathcal{A}}

\newcommand{\A}{\mathcal{A}}
\newcommand{\La}{\Lambda}
\newcommand{\la}{\lambda}

\newcommand{\e}{\varepsilon}
\newcommand{\dd}{\partial}

\newcommand{\sse}{\subset}
\newcommand{\lr}{\rightarrow}

\newcommand{\Br}{\operatorname{Br}}

\newcommand{\Aut}{\operatorname{Aut}}

\newcommand{\GL}{\operatorname{GL}}

\newcounter{daggerfootnote}

\newcommand{\bR}{\mathbb{R}}

\newcommand{\cH}{\mathcal{H}}

\newenvironment{RC}{\color{red}{\bf RC:} \footnotesize}{}

\begin{document}
	
	\title{On Newton polytopes of Lagrangian augmentations}
	\subjclass[2010]{Primary: 53D10. Secondary: 53D15, 57R17.}
	
	\author{Orsola Capovilla-Searle}
	\address{University of California Davis, Dept. of Mathematics, Shields Avenue, Davis, CA 95616, USA}
	\email{ocapovillasearle@ucdavis.edu}

	\author{Roger Casals}
	\address{University of California Davis, Dept. of Mathematics, Shields Avenue, Davis, CA 95616, USA}
	\email{casals@math.ucdavis.edu}
	
\maketitle

\begin{abstract} 
This note explores the use of Newton polytopes in the study of Lagrangian fillings of Legendrian submanifolds. In particular, we show that Newton polytopes associated to augmented values of Reeb chords can distinguish infinitely many distinct Lagrangian fillings, both for Legendrian links and higher-dimensional Legendrian spheres. The computations we perform work in finite characteristic, which significantly simplifies arguments and also allows us to show that there exist Legendrian links with infinitely many non-orientable exact Lagrangian fillings. 
\end{abstract}
\setcounter{tocdepth}{1}

\color{black}
\section{Introduction}\label{sec:intro}


The object of this note is to start exploring Newton polytopes associated to augmentations of the Legendrian dg-algebra. We show that integer point counts for the Newton polytopes of augmented Reeb chords are an effective invariant that can distinguish infinitely many fillings. Two advantages are that computations with Newton polytopes are significantly simpler and more robust than previously employed methods, such as monomial counts, and that they can distinguish infinitely many augmentations even in finite characteristic. We also present explicit computations and applications, including the first construction of a Legendrian sphere in all dimensions with infinitely many Lagrangian fillings, and the first construction of Legendrian links with infinitely many non-orientable Lagrangian fillings; in both cases fillings are distinguished by studying Newton polytopes associated to the corresponding augmentations and working in characteristic two.

\subsection{Scientific context} Exact Lagrangian fillings of Legendrian submanifolds have been a central object of study in contact and symplectic topology~\cite{ekholm_etnyre_sabloff_2009, sabloff_traynor_2013,hayden_sabloff_2014,ehk,Pan,Casals_gao}. Recently, ~\cite{Casals_gao} constructed Legendrian links in the contact Darboux 3-ball $(\R^3, \xi_{std})$ with infinitely many orientable exact Lagrangian fillings, distinct up to compactly supported Hamiltonian isotopy. Further examples of such Legendrians links in $(\R^3, \xi_{std})$ were found in~\cite{Gao_shen_weng_1,CGGLSS, Casals_weng,Casals_zaslow} using methods from cluster algebras and the  microlocal theory of sheaves.

The first, and currently only, Floer-theoretic method to distinguish such infinite families was introduced in \cite{CasalsNg} by L.~Ng and the second author. The techniques developed in \cite{CasalsNg}, especially the effective use of the $E$-invariant in \cite[Section 6]{CasalsNg}, require the Legendrian dg-algebra and the augmentations induced by Lagrangian fillings to be both constructed and computable over $\Z$. This involves a fair amount of technical details and it is desirable to have a simpler, more directly computable, Floer-theoretic way to distinguish the infinite families of augmentations coming from $\vartheta$-loops. The current note presents Newton polytopes as a versatile and effective alternative.

In \cite{Pan}, Y.~Pan used the count of monomials of $\Z_2[H_1(L;\Z)]$-augmented values of Reeb chords to distinguish a Catalan number of Lagrangian fillings $L\sse(\D^4,\la_{st})$ for the max-tb $(2,n)$ torus links $\La\sse (\R^3, \xi_{std})$. The same counts of such monomials for the Lagrangians associated to $\vartheta$-loops are more challenging and, even when computable, typically not able to show that the $\vartheta$-orbits are entire. Newton polytopes associated to multivariable Laurent polynomials are often more robust invariants than the precise count of monomials and contain a significant amount of geometric information, see e.g.~ \cite{Newton3,Newton5,Newton6,Newton4,Newton1,Newton2}. Through a variety of explicit computations and applications, this note will illustrate that Newton polytopes are useful in the study of the Legendrian contact dg-algebra.

\subsection{Main results}

Let $(\mathcal{A}(\Lambda;R), \partial_{\Lambda})$ denote the Legendrian contact dg-algebra of a Legendrian $\Lambda \subset  (\R^{2n+1}, \xi_{std})$ over a commutative unital ground ring $R$, see \cite{Chekanov,ekholm_etnyre_sullivan_2005,CasalsNg}. An exact Lagrangian filling $L$ of $\Lambda$ induces a dg-algebra map $\e_L:(\mathcal{A}(\Lambda; R), \partial_{\Lambda})\rightarrow (R, 0)$, referred to as the augmentation $\e_L$ associated to $L$. Two Hamiltonian isotopic exact Lagrangian fillings induce dg-algebra homotopic augmentations, cf.~ \cite{ehk,karlsson}. See Section \ref{sec:prelim} for more details.

First, consider a Legendrian link $\Lambda \subset  (\R^3, \xi_{std})$ and the $\vartheta$-loops introduced by L.~Ng and the second author in \cite{CasalsNg}. Denote by $\mathcal{A}(\vartheta)\in \mbox{Aut}(\mathcal{A}(\Lambda;R), \partial_{\Lambda})$ the dg-algebra automorphism induced by a Legendrian $\vartheta$-loop based at $\Lambda$. By definition, the $\vartheta$-orbit of an augmentation $\e_L$ induced by an exact Lagrangian filling $L$ of $\La$ is said to be entire if $\e_L \circ\mathcal{A}(\vartheta)^n\not\simeq \e_L \circ\mathcal{A}(\vartheta)^m$ for distinct $n,m\in \N$. In such a case, the Lagrangian fillings of $\La$ obtained by concatenating $L$ with a power of the Lagrangian trace of $\vartheta$ are pairwise not compactly supported Hamiltonian isotopic.

Let us focus on two of the smallest Legendrian links for which entire $\vartheta$-orbits are detected over $\Z$, as established in \cite{CasalsNg}. Let $\Lambda(\beta_{11})$, resp.~$\Lambda_1$, denote the Legendrian link whose $(-1)$-closure is the positive braid $(\sigma_2 \sigma_1 \sigma_3 \sigma_2)^4\sigma_1\sigma_3$, resp.~$(\sigma_2 \sigma_1 \sigma_1 \sigma_2)^3\sigma_1$. See Figure~\ref{fig:BraidIntroLoop} for a schematic of the Legendrian isotopy loop $\vartheta$ of $\Lambda(\beta_{11})$ that we consider. Our first result shows that, even over a ring of finite characteristic, integral point counts of Newton polytopes associated to the augmentations in the $\vartheta$-orbit pairwise distinguish them:

\begin{thm}\label{thm:main} Let $\Lambda\subset  (\R^3, \xi_{std})$ be either $\Lambda(\beta_{11})$ or $\Lambda_1$. Then there exists a $\vartheta$-loop, a Reeb chord $\rho\in\mathcal{A}(\Lambda;\Z_2)$ and an orientable exact Lagrangian filling $L$ with augmentation
$$\e_{L}:(\mathcal{A}(\Lambda;\Z_2),\partial_{\Lambda})\rightarrow (\Z_2[H_1(L;\Z)],0),$$
such that for any $n\in\N$ the Newton polytopes in $\R^{b_1(L)}$ of the augmented values $(\e_L \circ\mathcal{A}(\vartheta)^n)(\rho)$ are pairwise non-equivalent, even up to a $\GL(b_1(L),\Z)$ transformation, where $b_1(L)=\mbox{rk}(H_1(L;\Z))$ denotes the first Betti number of $L$.

In particular, the $\vartheta$-orbit of the augmentation $$\e_{L}:(\mathcal{A}(\Lambda;\Z_2),\partial_{\Lambda})\rightarrow (\Z_2[H_1(L;\Z)],0),$$
is entire, even over the characteristic 2 ring $\Z_2[H_1(L;\Z)]$.
\end{thm}

\begin{center}
	\begin{figure}[ht]
		\centering
		\includegraphics[scale=1]{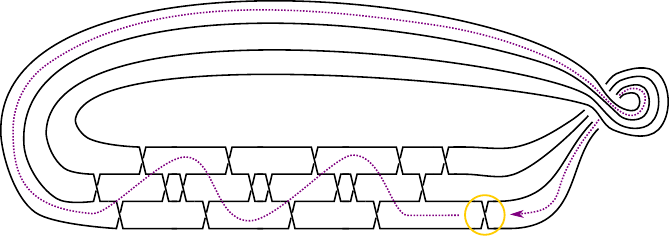}
		\caption{The Lagrangian projection of the Legendrian link $\Lambda(\beta_{11})$ associated to the 4-stranded braid $(\sigma_2\sigma_1\sigma_3\sigma_2)^4\sigma_3\sigma_1$. The $\vartheta$-loop for this Legendrian link is obtained by moving the rightmost $\sigma_1$ crossing around the braid, as indicated by the purple dashed arrow.}
		\label{fig:BraidIntroLoop}
	\end{figure}
\end{center}

First, the Newton polytopes in Theorem \ref{thm:main} are distinguished by their integral point counts, which are invariant under linear integral transformations. These Newton polytope invariants are simpler to compute than the $E$-invariant used in \cite{CasalsNg} and they have the additional advantage of working over finite characteristic. Second, in these instances, monomial counts of the Laurent polynomials $(\e_L \circ\mathcal{A}(\vartheta)^n)(\rho)$ do not pairwise distinguish them; see Subsection \ref{sec:proofmain}. In a sense, the Newton polytope of $(\e_L \circ\mathcal{A}(\vartheta)^n)(\rho)$ seems to be a more robust quantity than its monomial count.

It follows from Theorem \ref{thm:main} that $\Lambda(\beta_{11})$ and $\Lambda_1$ have infinitely many distinct exact Lagrangian fillings distinguished by augmentations over characteristic $2$. These links were first studied in \cite{CasalsNg} and are currently some of the smallest Legendrian links known to have infinitely many exact Lagrangian fillings in terms of genus and number of link components. Theorem \ref{thm:main} can then be used to prove that many more Legendrian links have infinitely many Lagrangian fillings that are also distinguished over finite characteristic:

\begin{thm}\label{thm:inf}
Let $\Lambda\subset  (\R^3, \xi_{std})$ be a Legendrian link with no Reeb chords of degree -1. Suppose that there exists an exact decomposable Lagrangian cobordism from either $\Lambda(\beta_{11})$ or $\Lambda_1$ to $\Lambda$. Then $\Lambda$ has infinitely many exact Lagrangian fillings distinguished by $\Z_2[H_1(L;\Z)]$-augmentations.
\end{thm}

A more general, if technical, version of Theorem~\ref{thm:inf} can be found in Section~\ref{sec:cobordisms}. In a nutshell, Newton polytopes are used in Theorems \ref{thm:main} and \ref{thm:inf} in the following manner. Consider an integral lattice $\Z^k\sse\R^k$, whose points we interpret as monomials in the variables $s_1^{\pm1}, \ldots, s_k^{\pm1}$, i.e.~the lattice is the character lattice of the torus associated to the Laurent polynomial ring in the variables $s_1^{\pm1}, \ldots, s_k^{\pm1}$. Consider a Laurent polynomial $f= \sum_{\alpha\in \Z^k} c_{\alpha}s^{\alpha}\in \Z_2[s_1^{\pm1}, \ldots, s_k^{\pm1}]$ in this ring; note that, by definition, $c_\alpha\neq0$ only for finitely many indices $\alpha\in\Z^k$. By definition, its Newton polytope $N(f)$ is the lattice polytope in $\R^k$ given by the convex hull of the exponent vectors in $f$, i.e.
$$N(f):=\mbox{Conv}\{ \alpha~|~c_{\alpha}\neq 0\}\subset \R^k.$$
For any two embedded exact Lagrangian fillings $L_1$ and $L_2$ of a Legendrian link $\Lambda \subset (\R^3, \xi_{std})$ with equal first betti number $k$, and for any Reeb chord $a$ of $\Lambda$, the induced augmentations $\e_{L_1}(a)$ and $\e_{L_2}(a)$ are Laurent polynomials in  $\Z_2[H_1(L;\Z)]\simeq \Z_2[s_1^{\pm1}, \ldots, s_k^{\pm1}]$. The group of unital $\Z_2[s_1^{\pm1}, \ldots, s_k^{\pm1}]$-automorphisms is isomorphic to $\GL(k,\Z)$. Therefore, it suffices to find $\GL(k;\Z)$ invariants of Laurent polynomials to distinguish $\Z_2[H_1(L;\Z)]$-valued augmentations $\e_L$ and thus their corresponding exact Lagrangian fillings. The $\GL(k;\Z)$-invariant that we use is the number of lattice points of the associated Newton polytopes.

Computing augmentations with $\Z_2[H_1(L,;\Z)]$ or $\Z_2[H_1(L,\Lambda;\Z)]$ coefficients, as opposed to computing them with $\Z[H_1(L, \Lambda;\Z)]$ or $\Z[H_1(L, \Lambda;\Z)]$, is simpler. It also has the advantage that we can then distinguish non-orientable exact Lagrangian fillings, which have been less studied than their orientable counterparts. See~\cite{Josh} for some results on the existence of non-orientable exact Lagrangian fillings. Section~\ref{sec:nonorientable} uses lattice invariants of Newton polytopes to distinguish sets of non-orientable exact Lagrangian fillings. In particular, we prove the following result:

\begin{thm}\label{thm:nonorientable} There exist Legendrian knots $\Lambda\subset(\R^3,\xi_{std})$ with infinitely many distinct non-orientable exact Lagrangian fillings. 
\end{thm}

In higher dimensions, front spinning previous examples of the second author, in \cite{Casals_gao} or \cite{CasalsNg}, produces higher-dimensional Legendrian submanifolds $\La\sse(\R^{2n+1},\xi_{st})$ with infinitely many Lagrangian fillings, see~\cite[Theorem 1.1]{Golovko}. Nevertheless, these Legendrians are never spheres. Two disadvantages of this are that, first, they are not topologically as simple as possible and, second, it is not possible to attach a Weinstein handle (resp.~ perform contact surgery) along them to produce interesting symplectic manifolds (resp.~ contact manifolds). It is therefore desirable to construct Legendrian spheres $\La\sse(\R^{2n+1},\xi_{st})$ in all dimensions with infinitely many Lagrangian fillings. We achieve this in Section \ref{sec:highdim}, also using Newton polytopes, where we prove the following result:

\begin{thm}\label{thm:highdim}
For any $n\geq 2$, there exist Legendrian spheres $\Lambda\sse(\R^{2n+1}, \xi_{std})$
with infinitely many exact Lagrangian fillings, distinct up to compactly supported Hamiltonian isotopy. In addition, these fillings are distinguished by augmentations over rings of characteristic two.
\end{thm}

Following \cite[Corollary 1.10]{Casals_gao}, attaching Weinstein $n$-handles along these Legendrian spheres produces $(2n+2)$-dimensional Weinstein manifolds which are homotopic to $S^{n+1}$ and have infinitely many closed embedded exact Lagrangian submanifolds, distinct up to Hamiltonian isotopy in the Weinstein manifold but smoothly isotopic.

Finally, a natural question is which lattice polytopes can occur as the Newton polytopes we consider. That is, which lattice polytopes are Newton polytopes of $\e_L(\rho)$, where $L$ is an embedded exact Lagrangian filling of a Legendrian $\Lambda$ and $\rho$ a Reeb chord in $\La$. A priori, it is not even clear if and when top-dimensional polytopes, i.e.~of dimension equal to $b_1$ of a Lagrangian filling, and which ones, appear as Newton polytopes of augmented Reeb values. As a starting step, we show in Proposition~\ref{prop:torusknots} that the max-tb Legendrian knots $\La(2,n)\sse(\R^3, \xi_{std})$, presented as the rainbow closure of $\sigma_1^n$, $n\in\N$ odd, are such that for any degree zero Reeb chord $\rho$ of $\La(2,n)$ there exist an augmentation $\e_L$ -- induced by a Lagrangian filling $L$ -- such that the Newton polytope $N(\e_L(\rho))$ is the top-dimensional standard $(n-1)$-simplex in $\bR^{n-1}$.

{\bf Acknowledgements:} 
O.~Capovilla-Searle is supported by the NSF Postdoctoral Research Fellowship DMS-2103188. R.~Casals is supported by the NSF CAREER DMS-1942363 and a Sloan Research Fellowship of the Alfred P. Sloan Foundation. We thank James Hughes, Lenhard Ng, Shanon Rubin, Daping Weng, and Angela Wu for helpful conversations.
\hfill$\Box$\\


\section{Preliminaries on the Legendrian dg-algebra}\label{sec:prelim}


In this section we present the necessary preliminaries on the Legendrian contact dg-algebra, in line with \cite{Chekanov,ehk,CasalsNg}.

\subsection{Legendrian links and Lagrangian fillings}\label{ssec:LinksFillings}

Let $(\R^{2n+1}, \xi_{std})$ denote the standard contact Darboux ball, with $\xi_{std}=\ker\{dz-\sum_{i=1}^n y_i dx_i\}$. By definition, a Legendrian submanifold is a smooth $n$-dimensional submanifold such that $T_p\Lambda \subset \xi_p$ for all $p\in \Lambda$. See \cite{ArnoldSing} for more background on contact structures and their Legendrian submanifolds.

\begin{definition}\label{defn:cobord} 
	Let $\Lambda_{+},\Lambda_{-}\sse(\R^{2n+1}, \xi_{std})$ be Legendrian submanifolds.  
	An exact Lagrangian cobordism $L$ from $\Lambda_{-}$ to $\Lambda_{+}$ is an embedded Lagrangian $(n+1)$-dimensional submanifold in the symplectization $(\R_t\times \R^{2n+1}, d(e^t(dz-\sum_{i=1}^n y_i dx_i)))$
	   that  has  cylindrical ends  and is exact in the following sense:  
for some $N>0$, 
	\begin{enumerate}
		\item  $L \cap ([-N,N]\times Y)$ is compact,
		\item  $L \cap ((N,\infty)\times Y)=(N,\infty)\times \Lambda_{+}$ and $L\cap ((-\infty,-N)\times Y)=(-\infty,-N)\times \Lambda_{-}$, 
	\item there exists  a function $f: L \rightarrow \R$  and constants $\mathfrak c_\pm$ such that 
		$e^t\alpha|_{L} = df$, where $f|_{(-\infty, -N) \times \Lambda_{-}} = \mathfrak c_{-}$, and $f|_{(N, \infty) \times \Lambda_{+}} = \mathfrak c_{+}$.
	\end{enumerate}
	By definition, an exact Lagrangian filling  $L$ of a Legendrian link $\Lambda$ is an exact Lagrangian cobordism from the empty set $\emptyset$ to the Legendrian link $\Lambda$.\hfill$\Box$
\end{definition}

Exact Lagrangian fillings can be studied via Legendrian invariants that are functorial over exact Lagrangian cobordisms, such as the Legendrian differential graded algebra $(\mathcal{A}(\Lambda;R), \partial_{\Lambda})$, see~\cite{Chekanov} for $\Lambda \subset (\R^{3}, \xi_{std})$ and ~\cite{ekholm_etnyre_sullivan_2007} for generalizations; see~\cite{Etnyre_ng} for a survey. The underlying algebra $\mathcal{A}(\Lambda;R)$ is a unital graded algebra over a unital coefficient ring $R$ that is generated, freely commutatively, by the set of Reeb chords of $\Lambda$. A Reeb chord of $\Lambda$ is a trajectory of the Reeb vector field, in this case $\partial_z$, that begins and ends on $\Lambda$. The grading on $\A(\Lambda;R)$ is defined on the Reeb chord generators via a Conley-Zehnder index, and is defined on a word as the sum of the gradings of the letters in the word. The differential $\partial_{\Lambda}$ counts rigid $\mathcal{J}$-holomorphic disks in the symplectization with boundary on $\R \times \Lambda$. See~\cite{Chekanov,Etnyre_ng} for further details. 

For a Legendrian link $\Lambda\subset (\R^3, \xi_{std})$, the Legendrian dg-algebra  $(\mathcal{A}(\Lambda;R), \partial_{\Lambda})$ can be computed combinatorially from a generic Lagrangian projection of $\Lambda$, where the Lagrangian projection is given by $\Pi_{xy}(x,y,z)=(x,y)$. The other useful projection is the front projection, given by $\Pi_{xz}(x,y,z)=(x,z)$; the translation from a front  projection of a Legendrian link to a Lagrangian projection can performed via ~\cite{ng_2003}. 
The Reeb chords of $\Lambda$ correspond to crossings in the Lagrangian projection, and the rigid $\mathcal{J}$-holomorphic disks that the differential counts are immersed punctured disks $\Delta\subset \R^2$ with $\partial(\Delta)\subset \Pi_{xy}(\Lambda)$ where the punctures map to convex corners of the crossings; see \cite{Etnyre_ng}. 


\begin{definition}
Let $\ell\in\N$ divide $2rot(\Lambda)$, where $rot(\Lambda)$ denotes the rotation number of $\Lambda$. An $\ell$-graded augmentation is a dg-algebra map $\e:(\A(\Lambda;R), \partial_{\Lambda})\rightarrow (R,0)$ such that $\e$ maps all elements of degree not divisible by $\ell$ to $0$. If $\ell=1$ then $\e$ is said to be ungraded, and if $\ell=0$ then $\e$ is said to be graded.\hfill$\Box$
\end{definition}

In this note we only use $\ell$-graded augmentations with $\ell=0,1$.

Exact Lagrangian cobordisms $L$ between Legendrians $\Lambda_{\pm}\subset (\R^{3}, \xi_{std})$ induce dga (dg-algebra) maps
$$\Phi_{L}: (\A(\Lambda_+;  \Z_2[H_1(L;\Z)]), \partial_{\Lambda_+})\rightarrow (\A(\Lambda_-; \Z_2[H_1(L;\Z)]), \partial_{\Lambda_-}).$$
If $L$ is an exact Lagrangian filling of a Legendrian $\Lambda\subset \R^{3}_{std}$ with Maslov number $\ell$, then $L$ induces an $\ell$-graded augmentation
$$\e_L:(\A(\Lambda;  \Z_2[H_1(L;\Z)]), \partial_{\Lambda})\rightarrow (\Z_2[H_1(L;\Z)], 0),$$
where $\Z_2[H_1(L;\Z)]$ is in grading $0$ modulo $\ell$ and it has trivial differential as a dg-algebra. This was first shown for $\Z_2$ coefficients in~\cite[Theorem 1.2]{ehk}, and later upgraded to $\Z_2[H_1(L;\Z)]$ coefficients in~\cite{cornwell_ng_sivek_2016,Pan}. Moreover, Karlsson observed in~\cite[Section 2.2]{karlsson} that this result can be extended to arbitrary exact Lagrangian cobordisms of Legendrians in the higher-dimensional Darboux balls. Adding in the hypothesis that $L$ is an oriented, spin exact Lagrangian, \cite{karlsson} extended the dga maps to $\Z[H_1(L;\Z)]$-coefficients. In line with \cite[Definition 3.9]{CasalsNg}, we can define:

\begin{definition}
A $k$-system of $\ell$-graded augmentations of a Legendrian link $\Lambda$ is a unital dg-algebra map
$$\e: (\A(\Lambda; R), \partial_{\Lambda}) \rightarrow (R[s_1^{\pm1}, \ldots, s_k^{\pm1}],0),$$
so that $\e \circ \partial_{\Lambda}=0$ and $\e(1)=1$, and such that $\e$ maps all elements of degree not divisible by $\ell$ to $0$. Two $k$-systems of $\ell$-graded augmentations 
$$ \e: \A(\Lambda;R) \rightarrow R[s_1^{\pm1}, \ldots, s_k^{\pm1}]~\text{and}~ \e ': \A(\Lambda;R) \rightarrow R[s_1'^{\pm1}, \ldots, s_k'^{\pm1}]$$
are equivalent if there exists an $R$-linear automorphism $$\phi: R[s_1^{\pm1}, \ldots, s_k^{\pm1}] \rightarrow R[s_1'^{\pm1}, \ldots, s_k'^{\pm1}]$$ such that $\e '=\phi \circ \e$.\hfill$\Box$
\end{definition}

We will need the group of automorphisms of the $R$-algebra of Laurent polynomials for $R$ a unital domain. The units of $R[\Z^k]$ are well-understood, e.g.~see \cite[Theorem 3.(b)]{Neher09}. Since any automorphism of $R[\Z^k]$ must send the generator of each $\Z$-summand to a unit, the automorphism group of $R[\Z^k]$ is as follows.

\begin{lemma}\label{lem:equivalence}
Let $R$ be a unital integral domain and $R^*$ its group of units. The group of unital $R$-linear automorphisms of $R[s_1^{\pm1}, \ldots, s_k^{\pm1}]$ is isomorphic to the semidirect product $(R^*)^k\rtimes_{\psi} \GL(k;\Z)$, where $$\psi:\GL(k;\Z)\rightarrow Aut((R^*)^k),\quad \psi_{A}(\alpha_1, \ldots, \alpha_k)=(\alpha_1^{a_{11}}, \ldots, \alpha_k^{a_{kk}}),\quad A=(a_{ij})\in \GL(k,\Z).$$
\end{lemma}

\begin{remark}
Therefore, the group of $\Z_2$-linear automorphisms of $\Z_2[s_1^{\pm1}, \ldots, s_k^{\pm1}]$ is $\GL(k;\Z)$. For $R=\Z_2[H_1(\Lambda)]\simeq \Z_2[t_1^{\pm1}, \ldots,t_l^{\pm1} ]$, the group of $R$-linear automorphisms of $R[s_1^{\pm1},  \ldots, s_k^{\pm1}]$ is $(\{t_{1}^{\alpha_1}\cdots t_{l}^{\alpha_l}~|~\alpha_1,\ldots,\alpha_l \in \Z\})^k \rtimes_{\psi} \GL(k;\Z)$. In this note $R=\Z_2$ suffices for our computations, but they could be upgraded to $R=\Z_2[H_*(\Lambda)]$. We will not consider base points on $\La$ in this note.\hfill$\Box$
\end{remark}

From now onwards $R$ will always denote a unital integral domain. Let $L_1,L_2$ be two exact Lagrangian fillings of a Legendrian $\Lambda\sse(\R^{2n+1},\xi_{std})$ such that $b_1(L_1)=b_1(L_2)=k$, for some $k\in\N$. These each induce a $k$-system of augmentations $\e_{L_1}$ and $\e_{L_2}$.\footnote{For discussion on base points, see \cite{CasalsNg}. These are often used if $R=\Z$, as in that manuscript. That said, we focus on $R=\Z_2$ in this note and we do not need base points.} If $n=1$, and both $L_1$ and $L_2$ are orientable, $b_1(L_1)=b_1(L_2)$ is always satisfied since all orientable exact Lagrangian fillings of a Legendrian link $\Lambda$ in the symplectization of $(\R^3,\xi_{std})$ realize the smooth $4$-ball genus of the smooth link type of $\Lambda$, see~\cite[Theorem 1.4]{chantraine}. If $L_1$ and $L_2$ are compactly supported Hamiltonian isotopic, the isotopy between $L_1$ and $L_2$ induces an invertible map $\phi:H_1(L_1;\Z)\rightarrow H_1(L_2;\Z)$. In addition, it is shown in~\cite[Theorem 1.3]{ehk} that the augmentations induced by Hamiltonian isotopic $L_1$ and $L_2$ are chain homotopic. For $k$-systems of graded augmentations $\e_1,\e_2$ to the corresponding group rings $R[H_1(L_i,\Z)]$, $i=1,2$, we have that $\phi \circ \e_1$ and $\e_2$ are chain homotopic. That is, there exists a degree 1 chain map $K:(\mathcal{A}(\Lambda;R), \partial_{\Lambda})\rightarrow(R[H_1(L_2;\Z)], 0)$ such that $\phi\circ\e_1-\e_2=K\circ \partial$, fitting into the following commutative diagram.
\begin{center}
\begin{tikzcd}
C_{-2} \arrow{rd}[inner sep=1pt]{K} & C_{-1} \arrow{l}{\partial}  \arrow{rd}[inner sep=1pt]{K} & C_0\arrow{l}{\partial}   \arrow[d, shift left=1ex, "\phi\circ \e_1"] \arrow[d,shift right=2ex, "\e_2"]\arrow{rd}[inner sep=1pt]{K} & C_1 \arrow{l} & \arrow{l}\\
0 & 0 \arrow{l} & R[H_1(L;\Z)] \arrow{l} & 0 \arrow{l}& \arrow{l}
\end{tikzcd}
\end{center}

If there are no $(-1)$ degree Reeb chords and $\e_1,\e_2$ are chain homotopic, then $\phi \circ \e_1=\e_2$, where $\phi$ is an $R$-linear automorphism of $R[H_1(L;\Z)]$; see ~\cite[Lemma 3.5]{Pan}. In the general case, in the presence of $(-1)$ degree Reeb chords, $\e_1$ being chain homotopic to $\e_2$ implies that $(\phi \circ \e_1)(a)=\e_2(a)$, for any cycles $a$ in $C_0$. The above diagram and discussion cover the case of graded augmentations. The case of ungraded augmentations is parallel, except the bottom row in the diagram above should just have $R[H_1(L;\Z)]$, as maps are not required to preserve degree. In summary, the result we use, implied by the discussion above, is the following:

\begin{lemma}\label{lem:cycle}
Let $\La\sse(\R^3,\xi_{std})$ be a Legendrian link and $L_1,L_2$ exact Lagrangian fillings of $\Lambda$, each inducing $k$-system of augmentations $\e_{L_1},\e_{L_2}$. Suppose that $a\in(\A(\Lambda; R), \partial_{\Lambda})$ is a degree 0 cycle and that $L_1$ and $L_2$ are compactly supported Hamiltonian isotopic. Then there exists an automorphism $\phi\in (R^*)\rtimes_{\psi} \GL(k;\Z)$ such that $(\phi \circ \e_1)(a)=\e_2(a).$
\end{lemma}

The following notion is useful when studying infinitely many Lagrangian fillings:

\begin{definition}\label{defn:auginf}
Let $\cH\sse\R^k$ be a rank $k$ integral lattice. A Legendrian submanifold $\Lambda$ is $R[\cH]$ aug-infinite if the collection of augmentations $\e: (\mathcal{A}(\Lambda;R), \partial_{\Lambda})\rightarrow (R[\cH],0)$ induced by Lagrangian fillings $L$, ranging over all possible such fillings $L$ such that $H_1(L,\Z)\simeq\cH$, is infinite. Here we consider two systems of augmentations to be equivalent, and thus contribute one to the above collection, if they are dga homotopic.\hfill$\Box$
\end{definition}

By the discussion above, an $R[\cH]$ aug-infinite Legendrian submanifold $\Lambda$ must have infinitely many Lagrangian fillings, distinct up to compactly supported Hamiltonian isotopy. See also the notion of $\Z$ aug-infinite in \cite[Definition 7.2]{CasalsNg}.


Finally, all Lagrangian fillings discussed in this article will be decomposable, see \cite[Section 6]{ehk} or \cite[Section 2]{CasalsNg}. In brief, an exact Lagrangian cobordism is said to be decomposable if it is the concatenation of elementary exact Lagrangian cobordisms associated to Legendrian isotopies, pinch moves, and minimum cobordisms. See also \cite[Section 4]{CasalsNg} for a description of the dga maps induced by decomposable Lagrangian cobordisms. In particular, two Legendrians $\Lambda_{\pm}$ are related by a pinch move if $\Lambda_{-}$ is given by taking a (contractible or proper) crossing in the Lagrangian projection of $\Lambda_+$ and replacing it with its $0$-resolution. Figure~\ref{fig:orientable} shows orientable pinch moves and Figure~\ref{fig:nonorientable_pinch} shows non-orientable pinch moves. Note that not all Reeb chords of a Legendrian link can be pinched: they must be contractible or proper, see ~\cite[Definition 4.3]{CasalsNg}. If two Legendrians $\Lambda_-$ and $\Lambda_+$ are related by an orientable pinch move, then there exists an exact Lagrangian orientable saddle cobordism from $\Lambda_{-}$ to $\Lambda_+$. If they are instead related by a non-orientable pinch move, then there exists an exact Lagrangian non-orientable cobordism between $\Lambda_{-}$ and $\Lambda_{+}$ which is topologically an $\R\P^2$ with punctures union a collection of cylinders. Suppose $\Lambda_-$ and $\Lambda_+$ are related by a Legendrian isotopy, then the Lagrangian trace of this isotopy can be perturbed to be an exact Lagrangian cobordism from $\Lambda_-$ to $\Lambda_+$, cf.~\cite{ehk}. Finally, it is proven in \cite{eliashberg_polterovich} that any max-tb Legendrian representative of the unknot has a unique embedded exact Lagrangian disk filling, up to compactly supported Hamiltonian isotopy. A minimum cobordism from $\Lambda$ to $\Lambda \sqcup U$, where $U$ is Legendrian unlinked from $\La$, is the unlinked union of such an embedded exact Lagrangian disk filling of $U$ and the trivial  Lagrangian cylinder $\R \times \Lambda$. The name comes from the fact that it can be represented by a Lagrangian cobordism such that the projection onto the $\R$-factor of the symplectization has a unique critical point and it is a minimum.

\begin{remark}\label{rmk:dga_map_non_orientable}
The article \cite{CasalsNg} provides combinatorial descriptions of the dga map over $\Z[H_1(L;\Z)]$ induced by orientable pinch moves which can be computed from the Lagrangian projection of $\Lambda$; see~\cite[Section 4]{CasalsNg} for more details. The same dga map also holds for non-orientable pinch moves with minor modifications if we work over $\Z_2[H_1(L;\Z)]$ coefficients, as discussed in Section~\ref{sec:nonorientable}.\hfill$\Box$
\end{remark}

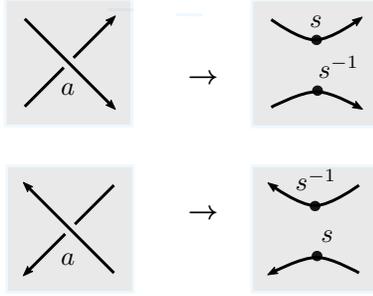
\begin{figure}
	\centering
	\begin{tikzpicture}[scale=1.5]
		\node[inner sep=0] at (0,0) {\includegraphics[width=5 cm]{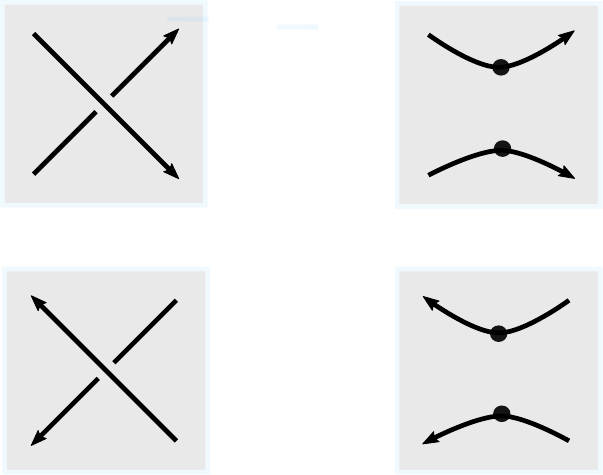}};
		\node at (-1.1,0.5){$a$};
		\node at (-1.1,-1){$a$};
		\node at (1.3,0.7){$s^{-1}$};
		\node at (1.1,1.1){$s$};
		\node at (1.1,-.3){$s^{-1}$};
		\node at (1.2,-0.8){$s$};
		\node at (0.1,0.6){$\boldsymbol{\rightarrow}$};
        \node at (0.1,-0.6){$\boldsymbol{\rightarrow}$};
	\end{tikzpicture}
	\caption{Orientable pinch moves at a Reeb chord $a$, where the crossing $a$ is replaced with its $0$-resolution and two basepoints $s, s^{-1}$. The basepoints are determined by the orientation of the arcs.} 
	\label{fig:orientable}
\end{figure}

\section{Newton polytopes distinguishing Lagrangian fillings}
\label{sec:proofs}



In this section we show that Newton polytopes can be used to distinguish augmentations induced by exact Lagrangian fillings of Legendrian links, in particular proving Theorems \ref{thm:main}, \ref{thm:inf} and \ref{thm:highdim}. Section~\ref{sec:newtonbackground} provides the necessary background on Newton polytopes. In Section~\ref{sec:lemmas} we collect the lemmas on powers of matrices with Laurent polynomial coefficients that will be used to compute Newton polytopes of augmentations in Section~\ref{sec:proofmain}. In Section~\ref{sec:proofmain} we prove Theorem~\ref{thm:main}. In Section~\ref{sec:cobordisms} we consider Legendrians related by exact Lagrangian cobordisms and prove Theorem~\ref{thm:inf}. Finally, in Section~\ref{sec:highdim} we study constructions of higher dimensional Legendrians and prove Theorem~\ref{thm:highdim}. Throughout this section $R$ is a base (unital) domain, which we take to be $R=\Z_2$ for all of our applications.

\subsection{Background on Newton polytopes.}\label{sec:newtonbackground}

Let $\cH=B(\Z^k)=\{Bx~|~x\in \Z^k\}$ be an integral lattice in $\R^k$, where $B$ is a basis of the $\R$-vector space $\R^k$. The collection $lin \cH =\{ Be_i\}_{i=1}^k$, where $e_i$ are the standard vectors of $\Z^k$, are said to be the linear spanning vectors of $\cH$. Given two integral lattices $\cH$ and $\cH'$, a linear map $\phi:lin\cH\rightarrow lin\cH'$ that induces a bijection from $\cH$ to $\cH'$ is said to be unimodular, a.k.a.~a lattice transformation. The group of lattice transformations for lattices of rank $k$ is the general linear group $\GL(k;\Z)$.

\begin{definition}
Consider an integral lattice $\Z^k\sse\R^k$ whose points we interpret as monomials in the variables $s_1^{\pm1}, \ldots, s_k^{\pm1}$; we denote this by $\Z^k\simeq\langle s_1^{\pm1}, \ldots, s_k^{\pm1}\rangle$. For any Laurent polynomial $f= \sum_{\alpha\in \Z^k} c_{\alpha}s^{\alpha}\in R[s_1^{\pm1}, \ldots, s_k^{\pm1}]$, its Newton polytope $N(f)$ is the lattice polytope in $\Z^k$ given by the convex hull of the exponent vectors in $f$, i.e.
$$N(f):=\mbox{Conv}\{ \alpha\in\Z^k|~c_{\alpha}\neq 0\}\subset \R^k.$$
Note that, by definition of a Laurent polynomial, only finitely many indices $\alpha~\in\Z^k$ have non-zero coefficient $c_{\alpha}\neq0$.
\hfill$\Box$
\end{definition}

Consider a $k$-system of augmentations $\e:(\A(\Lambda;R),\partial_{\Lambda})\rightarrow (R[s_1^{\pm1}, \ldots, s_k^{\pm1}], 0)$ and an element $a\in (\A(\Lambda;R),\partial_{\Lambda})$. The evaluation $\e(a)\in R[s_1^{\pm1}, \ldots, s_k^{\pm1}]$ is a Laurent polynomial. Let $N(\e, a)\sse\R^k$ denote the Newton polytope of this Laurent polynomial $\e(a)$. Note that if $\e=\e_L$ is induced by an exact Lagrangian filling $L$ of $\Lambda$, the lattice $\Z^k=\langle s_1^{\pm1}, \ldots, s_k^{\pm1}\rangle \cong H_1(L;\Z)$ is isomorphic to $H_1(L;\Z)$, which is free of rank $k$, for some $k\in\N$.

\begin{definition} Let $\cH,\cH'\sse\R^k$ be integral lattices. Two integral lattice polytopes $N\subset \cH$ and $N'\subset \cH'$ are said to be unimodular equivalent if there exists a lattice transformation $\phi\in \GL(k,\Z)$ such that $\phi(N\cap \cH)=N'\cap\cH'$, mapping the vertices of $N$ to those of $N'$. \hfill$\Box$
\end{definition}

\begin{definition}
An element $a\in(\mathcal{A}(\Lambda;R), \partial_{\Lambda})$ is said to be Newton infinite (over $R$) if the collection of integral lattice polytopes $N(\e,a)$, where $\e:\mathcal{A}(\Lambda;R)\lr R[H_1(L,\Z)]$ is induced by a Lagrangian filling $L$ of $\Lambda$, is infinite up to unimodular equivalence.\hfill$\Box$
\end{definition}

\begin{remark}\label{rem:newtoninfinite}
By Lemma~\ref{lem:cycle}, if an element $a$  of $(\mathcal{A}(\Lambda;\Z_2), \partial_{\Lambda})$ is Newton infinite and it is a cycle, i.e.~$\partial_{\Lambda}a=0$, then the Legendrian link $\Lambda$ is $\Z_2[\cH]$ aug-infinite for a lattice $\cH\simeq H_1(L,\Z)$, for some Lagrangian filling $L$ of $\La$.\hfill$\Box$
\end{remark}

\begin{remark}\label{rem:onereebchord}
Suppose that $\Lambda$ is $\Z_2[\cH]$ aug-infinite and has no $(-1)$ degree Reeb chords, so every $0$ graded Reeb chord is a cycle in the dga of $\Lambda$. Let $\{\e_i\}_{i\in \N}$ be a collection of $k$-systems of augmentations, induced by exact Lagrangian fillings of $\Lambda$, witnessing the aug-infinite property. Then, for every pair $i, j\in \N$, there exists a Reeb chord $x_{ij}$ of $\Lambda$ such that $N(\e_i,x_{ij})$ and $N(\e_j,x_{ij})$ are not unimodular equivalent. In addition, since $\Lambda$ has finitely many Reeb chords, there exists a subsequence of augmentations $\e_{i_l}$ such that for a single Reeb chord $x$ of $\Lambda$ the polytopes $N(\e_{i_l},x)$ and $N(\e_{i_{l'}},x)$ are not unimodular equivalent for $l\neq l'$. Thus, there is a Reeb chord $x$ that is Newton infinite.\hfill$\Box$
\end{remark}

The invariants of unimodular equivalence that we use are as follows. The following lemma is standard, e.g.~ proven in \cite[Corollary 2.75]{polytope}:

\begin{lemma}\label{lem:intpoints}
Let $\cH\sse\R^k$ be an integral lattice and $N\sse\cH$ an integral polytope. The
number of interior lattice points $\mbox{int}(N)\cap\cH$, the number of boundary lattice points $\dd N\cap\cH$, and the volume $\mbox{vol}(N)$ are invariant under unimodular equivalence.
\end{lemma}

\subsection{Linear algebra lemmata.}\label{sec:lemmas} We use the following lemmata to prove Theorem~\ref{thm:main}.

\begin{lemma}\label{lem:thirdvariable} Let $R$ be a commutative unital ring. Consider
$M_{i}=\left(\begin{array}{cc}
a_{i}&b_{i}\\
c_{i} & d_{i}
\end{array}\right)\in M_2(R)$ matrices over a commutative ring $R$, $1\leq i\leq r$, $r\in\N$. Fix a unit $u\in R$ and define new matrices $M'_{i}:=\left(\begin{array}{cc}
a_{i}&b_{i}u\\
c_{i}u^{-1} & d_{i}
\end{array}\right)$.

If the product $M_{1}\cdots M_{r}$ is $\left(\begin{array}{cc}
a&b\\
c & d
\end{array}\right)$, then the product $M_{1}'\cdots M_{r}'$ is $\left(\begin{array}{cc}
a&bu\\
cu^{-1} & d
\end{array}\right)$. \\In particular, the diagonal entries of these two products agree.
\end{lemma}
\begin{proof} Note that $M_i'$ is conjugate to $M_i$ via
$$M'_{i}=U^{-1}M_iU,\qquad\mbox{where }U:=\left(\begin{array}{cc}
1&0\\
0 & u
\end{array}\right).$$
Therefore $M_{1}'\cdots M_{r}'=U^{-1}\cdot M_{1}\cdots M_{r}\cdot U$ as required.
\end{proof}

\begin{lemma}\label{lem:convexhull}
For $n\geq 2$, the upper left entry of the $n^{th}$ power of the matrix
$$\begin{pmatrix}\tilde{x}+\tilde{y}&1\\
\tilde{x}&1\end{pmatrix}\in M_{2}(\Z_2[\tilde x, \tilde{y}])$$
is a polynomial $\alpha^n(\tilde{x},\tilde{y})\in \Z_2[\tilde{x}, \tilde{y}]$ containing terms $\tilde{x}, \tilde{x}^n,$ and $\tilde{y}^{n}$. In addition, the Newton polytope in $\Z^2\sse\R^2$ associated to $\alpha^n(\tilde{x},\tilde{y})$, when understood as a Laurent polynomial, is the convex hull of these three monomials.
\end{lemma}

\begin{proof}
Let us denote the matrix in the statement by $M$. The entries of $M^n$ are polynomials of degree (at most) $n$ in the entries of $M$ and therefore are themselves polynomials of degree at most $n$ in $\tilde x,\tilde y$. Let us simplify notation and write $x,y$ for $\tilde x,\tilde y$.

First, let us argue that the upper left entry of $M^n$ contains only one monomial of the form $y^k$, $0\leq k\leq n$, and it is precisely $y^n$, i.e. $k=n$ and with coefficient equal to one, even over $\Z$. Indeed, monomials of this form $y^k$ are the monomials remaining when we set $x=0$ in $M^n$. Therefore, they are the monomials in the upper left entry of the $n$th power of
$$M|_{x=0}=
\begin{pmatrix}y&1\\
0&1\end{pmatrix},\mbox{ which is }(M|_{x=0})^n=\begin{pmatrix} y^n & 1\\
0& 1\\
\end{pmatrix}.
$$
Second, let us argue that the monomials of the form $x^k$, $1\leq k\leq n$ in the upper left entry of $M^n$ always contain $x$ and $x^n$, i.e. the cases $k=1,n$ and both with coefficient equal to one. Again, monomials of this form $x^k$ are the monomials remaining when we set $y=0$ in $M^n$. Hence, they are the monomials appearing in the upper left entry of the $n$th power of

$$M|_{y=0}=
\begin{pmatrix}x&1\\
x&1\end{pmatrix},\mbox{ which is }(M|_{x=0})^n=
(1+x)^{n-1}\begin{pmatrix} x & 1\\
x& 1\\
\end{pmatrix}.
$$
Since the upper left entry is $x(1+x)^{n-1}$, it always contains the monomials $x$ and $x^n$. This proves the first statement: the polynomial $\alpha^n(x,y)\in \Z_2[x,y]$ contains always terms $x,x^n$ and $y^{n}$ with coefficient equal to one.

Finally, we note that the Newton polytope of $\alpha^n(x,y)\in \Z_2[x,y]$ is in fact the convex hull of (the points in $\Z^2$ associated to) $x,x^n$ and $y^n$. Indeed, we have proven that $\alpha^n(x,y)$ contains $x,x^n$ and $y^n$ and {\it does not} contain any term of the form $c\cdot y^k$, where $0\leq k\leq n-1$ and $c$ is a constant. In addition, all others monomials of $\alpha^n(x,y)$ are of degree at most $n$, and not powers of $y$, and thus lie in the convex hull of $x,x^n$ and $y^n$.
\end{proof}

Lemma \ref{lem:convexhull} and its proof illustrate an advantage of using Newton polytopes: there is no need to compute the entire Laurent polynomial entries, it suffices to keep track of monomials which generate the convex hull. In the particular matrix above, the number of Laurent monomials per coefficient varies as a function of the power $n$ but it suffices to focus only on the 3 monomials $\tilde x,\tilde x^n,\tilde y^n$ for each $n$.

\begin{cor}\label{cor:alpha} Let $i:\R^2_{\langle x,y\rangle}\lr \R^d$ be an embedding of the form $i(x,y)=(x,y,\iota(x,y))$, where $\iota(x,y):\R^2\lr\R^{d-2}$ is a linear map, $d\geq2$.
Let $n,m\in \N$ be two distinct numbers. Then there does not exist a change of coordinates $\phi \in \GL(d,\Z)$ such that $$\phi^*(i( \alpha^n(\tilde{x},\tilde{y})))=i(\alpha^m(\tilde{x},\tilde{y})).$$
\end{cor}
\begin{proof}
By Lemma~\ref{lem:convexhull}, the Newton polytope of $\alpha^n(\tilde{x},\tilde{y})$ is unimodular equivalent to the convex hull of the points $(1,0),(n,0),(0,n)$ in $\R^2_{\langle \tilde{x},\tilde{y}\rangle}$. Such a lattice polytope has $2n$ boundary lattice points and area $\frac{n}{2}(n-1)$. Pick's theorem implies that the number of interior lattice points is $\frac{1}{2}(n^2-3n)+1$. By Lemma~\ref{lem:intpoints}, the number of lattice points of a polytope in $\R^d$ is a $\GL(d,\Z)$ lattice invariant and thus $i(N(\alpha^n))$ is not unimodular equivalent to $i(N(\alpha^m))$ in $\R^d$ if $n\neq m$.
\end{proof}

We use the following two lemmas for the specific augmentations employed in Subsection \ref{sec:proofmain}. Both lemmas are consequences of Lemma \ref{lem:convexhull}.

\begin{lemma}\label{lemlambda11}
For $n\geq 2$, the upper left entry of the $n^{th}$ power of the matrix $$\begin{pmatrix}xy^{-1}+z^{-2}&y^{-1}\\
x&1\end{pmatrix}\in M_2(\Z_2[x^{\pm1}, y^{\pm1}, z^{\pm1}])$$ is a Laurent polynomial $\beta^n(x,y,z)\in \Z_2[x^{\pm1}, y^{\pm1}, z^{\pm1}]$ containing the terms $z^{-2n}, xy^{-1}$, and $x^ny^{-n}$, with coefficient 1. In addition, the Newton polytope associated to $\beta^n(x,y,z)$ is the convex hull of these three monomials.
\end{lemma}
\begin{proof}
By Lemma~\ref{lem:thirdvariable} applied to $u=y$, it suffices to consider instead the matrix $$\begin{pmatrix}
xy^{-1}+z^{-2}& 1\\
xy^{-1} &1
\end{pmatrix}.$$
Let $\tilde{x}=xy^{-1}$ and $\tilde{y}=z^{-2}$. By Lemma~\ref{lem:convexhull}, the upper left entry $\beta^n(x,y,z)$ of the $n$th power of this matrix is equal to $\alpha^n(xy^{-1}, z^{-2})$. Lemma~\ref{lem:convexhull} again implies that $\beta^n(x,y,z)$  must therefore be the convex hull of the three monomials $z^{-2n}$, $xy^{-1}$ and $x^ny^{-n}.$
\end{proof}

\begin{lemma}\label{lemlambda1}
For $n\geq 2$, the upper left entry of the $n^{th}$ power of the matrix $$\begin{pmatrix}z^{-2}+yx^{-1}z^{-1}&x^{-1}z^{-1}\\
    y&1\end{pmatrix}\in M_2(\Z_2[x^{\pm1}, y^{\pm1}, z^{\pm1}])$$ is a Laurent polynomial $\gamma^n(x,y,z)\in \Z_2[x^{\pm1}, y^{\pm1}, z^{\pm1}]$ containing terms $yx^{-1}z^{-1}, y^nx^{-n}z^{-n}$, and $z^{-2n}$, with coefficient one. In addition, the Newton polytope associated to $\gamma^n(x,y,z)$ is the convex hull of these three monomials.
\end{lemma}
\begin{proof}
By Lemma~\ref{lem:thirdvariable} applied to $u=xz$, we can consider instead the matrix $$\begin{pmatrix}
z^{-2}+yx^{-1}z^{-1}&1\\
    yx^{-1}z^{-1}&1
\end{pmatrix}.$$
Let $\tilde{x}=yx^{-1}z^{-1}$ and $\tilde{y}=z^{-2}$. By Lemma~\ref{lem:convexhull}, the upper left entry of the $n$th power of this matrix $\gamma^n(x,y,z)$ is equal to $\alpha^n(yx^{-1}z^{-1}, z^{-2})$, which is also the convex hull of the three monomials $yx^{-1}z^{-1}, y^nx^{-n}z^{-n}$ and $z^{-2n}$.
\end{proof}


In the same manner that Lemma \ref{lem:convexhull} implies Lemmas \ref{lemlambda11} and \ref{lemlambda1}, Corollary \ref{cor:alpha} implies:

\begin{cor}\label{cor:beta}
Let $i:\R^2_{\langle x,y\rangle}\lr \R^d$ be an embedding of the form $i(x,y)=(x,y,\iota(x,y))$, where $\iota(x,y):\R^2\lr\R^{d-2}$ is a linear map, $d\geq3$.
Let $n,m\in \N$ be two distinct numbers. Then there does not exist a change of coordinates $\phi \in \GL(d,\Z)$ such that $$\phi^*(i( \beta^n(x,y,z)))=i(\beta^m(x,y,z))\quad\mbox{ or } \quad\phi^*(i( \gamma^n(x,y,z)))=i(\gamma^m(x,y,z)).$$ \end{cor}

\begin{proof}
By construction, $\beta^n(x,y,z)=\alpha^n(xy^{-1},z^{-2})$ and the Newton polytope of $\beta^n(x,y,z)$ is given by the convex hull of the points $(1,-1,0), (n,-n,0),$ and $(0,0,-2n)$ in $\R^3_{\langle x,y,z\rangle}$. Consider $\Phi\in\Aut(\Z[x^{\pm1},y^{\pm1},z^{\pm1}])$ such that
$$\Phi(x)=xy, \Phi(y)=y,\Phi(z)=z^{-1}.$$
Then, $\Phi(N(\beta^{n}))$ is the convex hull of the points $(1,0,0), (n,0,0),$ and $(0,0,2n)$ in $\R^3_{\langle x,y,z\rangle}$ and it is contained in the $\langle x,z\rangle$ plane. Note that $\Phi(N(\beta^{n}))$ has $2n$ boundary lattice points and area $n^2-n$; by Pick's Theorem, $\Phi(N(\beta^{n}))$ has $n^2-n-1$ interior lattice points, and therefore has $n^2+n-1$ lattice points in total. By Lemma~\ref{lem:intpoints}, the number of lattice points of a polytope in $\R^d$ is a $\GL(d,\Z)$ invariant and thus $i(N(\beta^n))$ is not unimodular equivalent to $i(N(\beta^m(x,y,z)))$ if $n\neq m$.

Similarly, by construction $\gamma^n(x,y,z)=\alpha^n(yx^{-1}z^{-1},z^{-2})$ and its Newton polytope is given by the convex hull of the points $(-1,1,-1), (-n,n,-n),$ and $(0,0,-2n)$ in $\R^3_{\langle x,y,z\rangle}$. Choose the automorphism $\Phi\in\Aut(\Z[x^{\pm1},y^{\pm1},z^{\pm1}])$ such that
$$\Phi(x)=x^{-1}yz, \Phi(y)=y,\Phi(z)=z^{-1}.$$
Then, $\Phi(N(\beta^{n})$ is again the convex hull of the points $(1,0,0), (n,0,0),$ and $(0,0,2n)$ in $\R^3_{\langle x,y,z\rangle}$. As above, $i(N(\beta^n))$ is not unimodular equivalent to $i(N(\beta^m(x,y,z)))$ in $\R^d$ if $n\neq m$.
\end{proof}


\subsection{Proof of Theorem~\ref{thm:main}}\label{sec:proofmain}

Let $\Lambda(\beta_{11})$ denote the Legendrian link whose front projection is the $(-1)$-closure of the positive braid $(\sigma_2 \sigma_1 \sigma_3 \sigma_2)^4\sigma_1\sigma_3$. See Figure~\ref{fig:BraidIntroLoop} for a Lagrangian projection of $\Lambda(\beta_{11}).$ Similarly, $\Lambda_1$ denotes the $(-1)$-closure of the positive braid $(\sigma_2 \sigma_1 \sigma_1 \sigma_2)^3\sigma_1$. We now prove Theorem~\ref{thm:main} with the following strategy, the beginning of which aligns with \cite{CasalsNg}.

We first construct an exact, orientable, Lagrangian filling $L$ of $\Lambda \in \{ \Lambda(\beta_{11}), \Lambda_1\}$. By~\cite{ehk,CasalsNg}, $L$ induces a system of augmentations $\varepsilon_L:(\mathcal{A}(\Lambda; \Z_2), \partial_{\Lambda})\lr(\Z_2[H_1(L,\Z)],0)$. Then we consider a Legendrian loop $\vartheta: (S^1, pt) \rightarrow (\mathcal{L}(\Lambda), \Lambda)$, where $(\mathcal{L}(\Lambda), \Lambda)$ is the space of Legendrian links that are isotopic to $\Lambda$, with basepoint an arbitrary fixed representative of $\Lambda$. For both links we shall use the same Lagrangian projections and $\vartheta$-loops as in \cite[Section 4]{CasalsNg}. The exact Lagrangian concordance $N_\vartheta$ of $\Lambda$, given by the trace of $\vartheta$, induces a dga automorphism 
$\Phi: (\mathcal{A}(\Lambda; \Z_2), \partial_{\Lambda})\rightarrow (\mathcal{A}(\Lambda; \Z_2), \partial_{\Lambda})$ by invariance \cite{Chekanov}. 
The exact Lagrangian fillings $L^i$ of $\La$ given by composing $i$ copies of the concordance $N_\vartheta$ with the initial filling $L$ have corresponding augmentations $\e^i_L:=\varepsilon_L \circ \Phi^i$. The key step at this stage is distinguishing such systems of augmentations. Indeed, if these augmentations  $\e_L^i$, $i\in\N$, are pairwise distinct as systems of augmentations, the Lagrangians fillings $L^i$, $i\in\N$, are also pairwise distinct, up to compactly supported Hamiltonian isotopy. The proofs below will precisely show how to pairwise distinguish these systems of augmentations.

\begin{remark} For context, the crucial novelty of our approach is in pairwise distinguishing these augmentations $\e^i_L$, $i\in\N$. This had not been achieved over a ring of finite characteristic. We succeed at distinguishing them via studying Newton polytopes of augmented values of Reeb chords, augmented according to $\e^i_L$. As computations will show, this is a more direct and computable method than the solution by L.~Ng and the second author in \cite{CasalsNg}, which only worked over $\Z$.\hfill$\Box$\end{remark}

\begin{proof}[Proof of Theorem~\ref{thm:main}]
\textbf{Case 1: $\Lambda=\Lambda(\beta_{11})$}. Let us label the Reeb chords of $\Lambda(\beta_{11})$ by $a_1,a_2,\ldots,a_{16}$ for the crossings corresponding to the piece $(\sigma_2 \sigma_1 \sigma_3 \sigma_2)^4$ read right to left. That is, the rightmost $\sigma_2$ crossing is $a_1$, the next $\sigma_3$ crossing is $a_2$ and so on until the leftmost $\sigma_2$ crossings is $a_{16}$. The additional two crossings $\sigma_1\sigma_3$ on the right will be labeled $a_{19}$, for $\sigma_1$, and $a_{17}$, for $\sigma_3$. Even if this is an apparently odd labeling (with no $a_{18}$), this is the same notation as in Section $6.3$ in~\cite{CasalsNg} (where $a_{18}$ appeared elsewhere in that article) and it is helpful to keep this notation the same for our computations.

Consider the Lagrangian filling $L$ constructed by first pinching in order the following Reeb chords $a_9,a_{10},a_{11},a_{12},a_{13}$, and $a_{16}$ of $\Lambda(\beta_{11})$. After this pinching sequence, we have four max-tb unknots, so we compose the Lagrangian cobordism induced by the pinching sequence with four minimum cobordisms to construct the Lagrangian filling $L$. The four max-tb unknots have basepoints
$$s_{16}s_{10}, s_{16}^{-1}s_{11}^{-1}, s_{13}s_{12}s_{10}^{-1}s_{9},s_{13}^{-1}s_{12}^{-1}s_{11}s_{9}^{-1}.$$
The unique Lagrangian disk fillings of each these max-tb unknots induce an augmentation which maps these monomials to $1$. Therefore, we have the relations \begin{align*}
   s_{10}=s_{11}=s_{16}^{-1}=s_{13}s_{12}s_{9}.
\end{align*}

 We can identify $\Z_2[H_1(L;\Z)]\cong\Z_2[ s_9^{\pm1},s_{11}^{\pm1},s_{12}^{\pm1},s_{13}^{\pm1}]$ and the Lagrangian filling $L$ we have constructed induces the following augmentation  $\varepsilon_{L}$:
\begin{align*}
\varepsilon_{L}(a_{9})&=s_{9},~~~ \varepsilon_{L}(a_{11})=s_{11},\\ \varepsilon_{L}(a_{19})&=\frac{s_{9}s_{13}s_{12}^2}{s_{10}s_{11}^2s_{16}}+\frac{s_{13}s_{12}}{s_{11}s_{16}}+\frac{s_9}{s_{11}}\\
&=s_{12}s_{11}^{-1}+s_{11}s_{9}^{-1}+s_{11}^{-1}s_9.
\end{align*}
For $\varepsilon_{L}(a_{19})$, the first equality of
$$\frac{s_{9}s_{13}s_{12}^2}{s_{10}s_{11}^2s_{16}}=s_{12}s_{11}^{-1},\quad \frac{s_{13}s_{12}}{s_{11}s_{16}}=s_{11}s_{9}^{-1},$$
uses $s_{10}=s_{13}s_{12}s_{9}$ and $s_{11}=s_{16}^{-1}$. The second equality uses $s_{11}=s_{16}^{-1}$ and $s_{11}=s_{13}s_{12}s_{9}$. The computations of $\varepsilon_{L}(a_{9}),\varepsilon_{L}(a_{11})$ and $\varepsilon_{L}(a_{19})$ above can be performed directly iteratively using \cite[Subsection 4.2]{CasalsNg}, specifically \cite[Proposition 4.8]{CasalsNg}. These are similar to the computations performed in \cite[Section 6]{CasalsNg}, which closely follows the $\Z_2$-computations in \cite[Section 6.5]{ehk} and \cite[Section 3A]{Pan}. Alternatively, these computations can be performed with the same Mathematica package used in \cite{CasalsNg}, which is available in the website of the second author. Since we work in characteristic 2, we do not need to be concerned with spin structures, which computationally translated into additional basepoints in \cite{CasalsNg}.

Let $\vartheta$ be the Legendrian loop schematized in Figure~\ref{fig:BraidIntroLoop}. By \cite[Prop. 5.3]{CasalsNg}, this loop $\vartheta$ induces a dga automorphism $\mathcal{A}(\vartheta)$ which acts on the Reeb chords $a_{11}$ and $a_9$ by

$$ \mathcal{A}(\vartheta)\begin{pmatrix} a_{11}\\
a_9\end{pmatrix}= \begin{pmatrix} 0 & 1\\
1 &a_{19}\end{pmatrix} \begin{pmatrix} a_{11}\\
a_9\end{pmatrix}.$$

Therefore, the fillings $L^n$ of $\Lambda(\beta_{11})$ -- given by concatenating $n$ copies of the concordance $N_\vartheta$ and the filling $L$ -- induce the augmentations
$$\e_L^n:=\e_L\circ\mathcal{A}(\vartheta)^n:(\mathcal{A}(\Lambda; \Z_2), \partial_{\Lambda})\lr(\Z_2[H_1(L,\Z)],0).$$
In particular, 
\begin{align*}
    \begin{pmatrix} \e^n_{L}(a_{11})\\
\e^n_{L}(a_9) \end{pmatrix}=M^n  \begin{pmatrix} \e_L(a_{11})\\
\e_{L}(a_9)\end{pmatrix}:= \begin{pmatrix} 0 & 1\\
1 &\e_{L}(a_{19})\end{pmatrix}^n  \begin{pmatrix} \e_{L}(a_{11})\\
\e_{L}(a_9)\end{pmatrix}.
\end{align*}
where $M$ is the matrix obtained as the result of applying the augmentation $\e_L$ to the entries of the matrix
$$\begin{pmatrix} 0 & 1\\
1 &a_{19}\end{pmatrix},\quad\mbox{so that }M=\begin{pmatrix} 0 & 1\\
1 &\e_{L}(a_{19})\end{pmatrix}=
\begin{pmatrix} 0 & 1\\
1 &s_{12}s_{11}^{-1}+s_{11}s_{9}^{-1}+s_{11}^{-1}s_9
\end{pmatrix}.$$
Note that the matrix on the left is the matrix describing how $\mathcal{A}(\vartheta)$ acts on $\begin{pmatrix} a_{11}\\
a_9\end{pmatrix}$.

Let $M_1:=N^{-1}MN$ where $N:=\begin{pmatrix} s_{11} & 1\\
s_9 &0\end{pmatrix}$. Then $
      \begin{pmatrix} \e^n_{L}(a_{11})\\
\e^n_{L}(a_9) \end{pmatrix}= NM_1^n\begin{pmatrix} 1\\0\end{pmatrix},$ where 
\begin{align*}
    M_1&=\begin{pmatrix} s_{11}s_{9}^{-1}+\e_L(a_{19})&s_9^{-1}\\ s_9^{-1}s_{11}^2+s_9+s_{11}\e_L(a_{19}) &s_9^{-1}s_{11} \end{pmatrix}
   =s_{9}^{-1}s_{11} \begin{pmatrix}
   s_{12} s_{9}s_{11}^{-2}+s_9^2s_{11}^{-2} & s_{11}^{-1}\\
    s_{12}s_9s_{11}^{-1} & 1
    \end{pmatrix}.
\end{align*}

Let us set
$$x:=s_{12}s_9s_{11}^{-1},\quad y:=s_{11},\quad z:=s_{9}^{-1}s_{11}.$$
By Lemma~\ref{lemlambda11}, the upper left entry of $M_1^n$ is $z^n\beta^n(x,y,z)\in \Z_2[x^{\pm1}, y^{\pm1}, z^{\pm1}]$. Therefore
$$\e^n_L(a_{9})=s_{9}z^n\beta^n(x,y,z)=(yz^{n-1})\cdot\beta^n(x,y,z),$$
whose Newton polytope is a translate of the Newton polytope for $\beta^n(x,y,z)$. By Corollary~\ref{cor:beta}, $i(N(\beta^n(x,y,z)))$ and $i(N(\beta^m(x,y,z)))$ are not unimodular equivalent as lattice polytopes in $\R^4_{\langle x,y,z, s_{13} \rangle}$ if $n\neq m$. In particular, the Newton polytopes of $\e^{n}_L(a_9)$ and $\e^{m}_L(a_9)$ are also not unimodular equivalent $\R^{4}_{\langle x,y,z, s_{13}\rangle}$. Since $\partial_{\Lambda}a_9=0$, then $\e^{n}_L(a_9)$ and $\e^{m}_L(a_9)$ are not dga homotopic augmentations if $n\neq m$. Therefore, the $\vartheta$-orbit of $\e_L$ is entire over characteristic 2, as we wanted to show.


\textbf{Case 2: $\Lambda=\Lambda_1$}. Let the Reeb chords of $\Lambda_1$ corresponding to the crossings in the positive braid $(\sigma_2 \sigma_1 \sigma_1 \sigma_2)^3\sigma_1$ be labelled $a_1, \ldots, a_{13}$ from left to right. Consider the exact Lagrangian filling $L$ of $\Lambda_1$ given by pinching the Reeb chords $a_{9},a_{13},a_{10},a_{12},a_{4},a_1$, in this order. After performing this pinching sequence, there are three max-tb unknots with basepoints 
 $(s_{10}s_{11}s_{13})^{-1}$, $s_{10}s_{11}s_{13}(s_1s_4s_9s_{12})^{-1}$, and $s_1s_4s_9s_{12}.$ 
Therefore, we have the relations $s_{10}s_{11}s_{13}=1,$ and $s_1s_4s_9s_{12}=1$. We also identify $\Z_2[H_1(L;\Z)] \cong \Z_2[s_4^{\pm 1}, s_9^{\pm1}, s_{10}^{\pm1},s_{12}^{\pm1}, s_{13}^{\pm1}]$.

Similary to the case of $\La(\beta_{11})$ above, e.g.~ using the mathematica package from~\cite{CasalsNg}, we find the following $\e_L$-augmented values:
    \begin{align*}
\e_L(a_{9})&=s_9,\\
\e_L(a_{10})&=s_{10}+s_{10}^2s_{11}^2s_{13}(s_1s_4s_{12})^{-1}=s_{10}+s_9s_{13}^{-1},\\
\e_L(a_{13})&=s_{10}s_9^{-1}+s_{13}+s^2_{10}s_{11}^2s_{13}(s_1s_4s_9s_{12})^{-1}=s_{10}s_9^{-1}+s_{13}+s_{13}^{-1}.
    \end{align*}

Let $\vartheta=\vartheta_2$ denote the Legendrian $\vartheta$-loop of $\Lambda_1$ that takes the crossing $a_{13}$ -- the purple box is $a_{13}$ in the language of \cite{CasalsNg} -- and pushes it right to left across the entire links. By \cite[Prop 5.3]{CasalsNg}, its monodromy map $\mathcal{A}(\vartheta_2)$ acts on the Reeb chords $a_{9}$ and $a_{10}$ by:
\begin{align*}
   \mathcal{A}(\vartheta_2) \begin{pmatrix}
  a_{10}\\
 a_{9}
    \end{pmatrix}&= \begin{pmatrix}
    a_{13} &1\\
    1 & 0
    \end{pmatrix}\begin{pmatrix}   a_{10}\\
     a_{9}\end{pmatrix}.
     \end{align*}

The fillings $L^n$ of $\Lambda_1$, given by concatenating $n$ copies of the concordance $N_{\vartheta_2}$ and the filling $L$, induce augmentations
$$\e_L^n:=\e_L\circ\mathcal{A}(\vartheta_2)^n:(\mathcal{A}(\Lambda; \Z_2), \partial_{\Lambda})\lr(\Z_2[H_1(L,\Z)],0).$$
The augmented values for $a_{9}$ and $a_{10}$ are
\begin{align*}
    \begin{pmatrix} \e^n_L(a_{10})\\\e^n_L(a_{9})\end{pmatrix}=\begin{pmatrix}    \e_L(a_{13})&1\\
    1 & 0\end{pmatrix}^n\begin{pmatrix}   \e_L(a_{10})\\
    \e_L(a_{9})\end{pmatrix}.
\end{align*}

Let $N:=\begin{pmatrix}\e_L(a_{10})&1\\
\e_L(a_9)&0
\end{pmatrix}=\begin{pmatrix}s_{10}+s_9s_{13}^{-1}&1\\s_9&0
 \end{pmatrix}$, so that $N^{-1}=\begin{pmatrix}0&s_9^{-1}\\
 1&s_{13}^{-1}+s_{10}s_{9}^{-1}\end{pmatrix}.$
As before, set $M_1:=N^{-1}MN$ so as to have the equality $\begin{pmatrix} \e^n_L(a_{10})\\ \e^n_L(a_{9})
 \end{pmatrix}= NM_1^n \begin{pmatrix} 1\\0 \end{pmatrix}$.
Then
\begin{align*}
M_1=N^{-1}\begin{pmatrix}s_{10}s_9^{-1}+s_{13}+s_{13}^{-1} &1\\
1&0\end{pmatrix}N&=\begin{pmatrix} s_{13}^{-1}+s_{10}s_9^{-1}&s_9^{-1}\\
s_{10}s_{13}&s_{13}
\end{pmatrix}\\
&=s_{13}\begin{pmatrix} s_{13}^{-2}+s_{10}s_9^{-1}s_{13}^{-1}&s_9^{-1}s_{13}^{-1}\\
s_{10}&1\end{pmatrix}.
\end{align*}

Let us set $s_{9}=x$, $s_{10}=y,$ and $s_{13}=z$, so that the above matrix becomes
\begin{align*}
    M_1&:=z\begin{pmatrix} z^{-2}+yx^{-1}z^{-1}&x^{-1}z^{-1}\\
    y&1
    \end{pmatrix}.
\end{align*}
By Lemma~\ref{lemlambda1}, the upper left entry of $M_1^n$ is the Laurent polynomial $z^n\gamma^n(x,y,z)\in \Z_2[x^{\pm1}, y^{\pm1},z^{\pm1}].$  In addition,
$\e^n_L(a_{9})=\e_L(a_{9})z^n\gamma^n(x,y,z)=xz^n\gamma^{n}(x,y,z)$ and thus the Newton polytope $N(\e^n_L, a_{9})$ is a translate of the Newton polytope $N(\gamma^n(x,y,z))$. By Corollary~\ref{cor:beta}, the Newton polytopes of $i(N(\gamma^n(x,y,z)))$ and $i(N(\gamma^m(x,y,z)))$ are not unimodular equivalent in $\R^5_{\langle x,y,z, s_4,s_{12}\rangle}$ if $n\neq m$. Therefore, the Newton polytopes $N(\e^n_L, a_{9})$ and $N(\e^m_L, a_{9})$ are not unimodular equivalent in $\R^5_{\langle x,y,z, s_4,s_{12}\rangle}$ if $n\neq m$. Since $\partial_{\Lambda}a_9=0$, then $\e^{n}_L(a_9)$ and $\e^{m}_L(a_9)$ are not dga homotopic augmentations if $n\neq m$. Thus, the $\vartheta$-orbit of $\Lambda_1$ is entire over characteristic 2 as required.
\end{proof}

\begin{remark}~\label{rem:count_monomials}
For either $\Lambda(\beta_{11})$ or $\Lambda_1$, the number of monomials of $\e^n_L(a_9)$ is equal to the number of monomials of $\alpha^n$, where the count is modulo 2. Now, the number of monomials in $\alpha^n(\tilde{x},\tilde{y})\in \Z_2[\tilde{x}^{\pm1}, \tilde{y}^{\pm1}]$ is not an increasing, or non-decreasing sequence, nor is it true that for $n\neq m$ the number of monomials of $\alpha^n(\tilde{x},\tilde{y})$ is distinct from the number of monomials of $\alpha^m(\tilde{x},\tilde{y})$. Indeed for both $n=7$ and $n=10$, the number of $\alpha^n(\tilde{x},\tilde{y})$ is equal to $14$. The sequence of the number of monomials of $\alpha^n(\tilde{x},\tilde{y})$,  modulo $2$, for $n\leq 20$, is:
\begin{align*}2, 3, 5, 6, 7, 9, 14, 15, 13, 14, 19, 21, 22, 27, 41, 42, 31, 29, 34,
35, 33,\\ 38, 55, 57, 46, 47, 61, 66, 67, 81, 122, 123, 85, 74, 79, 77, 
66, 71, 97, 98.\end{align*}
In this case, Newton polytopes are a stronger $\GL(k,\Z)$-invariant of augmentations over characteristic 2 than the count of monomials. With the several computations we have performed, working for this article and previous work \cite{CasalsNg}, we have found this to be, at least heuristically, a general pattern. It is also the case that, also in general, computing lattice points of Newton polytopes is a much simpler process than counting monomials modulo $2$.\hfill$\Box$
\end{remark}


\subsection{Exact Lagrangian cobordisms}\label{sec:cobordisms}
Let us now prove Theorem~\ref{thm:inf}.

\begin{prop}\label{cor:cobordism_decomp}
Let $\Lambda_{\pm} \subset (\R^3, \xi_{std})$ be two Legendrian links such that there exists a decomposable Lagrangian cobordism $L$ from $\Lambda_-$ to $\Lambda_+$, both in the symplectization of $(\R^3, \xi_{std})$. Suppose that $\La_\pm$ have no $(-1)$-degree Reeb chords. In addition, suppose that there exists a Newton infinite element $\rho_-\in\SA(\La_-)$ such that there is a collection of Newton polytopes $N(\e,a)$ witnessing the Newton infinite property each with a different number of lattice points.

Then $\Lambda_+$ is $\Z_2[\cH\oplus\Z]$ aug-infinite, where $\cH\simeq H_1(L_-,\Z)$ for some Lagrangian filling $L_-$ of $\La_-$.
\end{prop}

\begin{proof}
It suffices to consider a single pinch move; then we have an isomorphism $\Z_2[H_1(L_+;\Z)]\simeq\Z_2[H_1(L_-;\Z), s^{\pm1}]$, where $s$ is the variable keeping track of the unstable arc of the unique saddle point of $L$. By~\cite{ehk}, the dga map induced by a pinch move at a contractible (or proper) Reeb chord $c$ of $\Lambda_+$ satisfies $\Phi_L(c)=s$. In addition, for any Reeb chord $a\neq c$ of $\Lambda_+$, we have $$\Phi_L(a)=a+w(a)s^{-1}$$ where $w(a)$ is either zero or some word of Reeb chords of $\Lambda_+$ not containing $a$ or $c$. 

Let $\Phi^{\Z_2}_L:(\mathcal{A}(\Lambda_+; \Z_2[H_1(L_-;\Z)]), \partial_{\Lambda_+})\rightarrow( \mathcal{A}(\Lambda_-; \Z_2[H_1(L_-;\Z)]), \partial_{\Lambda_-})$ denote the dga map $\Phi_L$ where $s=1$. Observe that $\Phi^{\Z_2}_L$ respects the height filtration, and so the Reeb chords in $w(a)$ all have height less than $a$. This implies that $\Phi^{\Z_2}_L$ is surjective. In particular, there exists $\rho_+\in \mathcal{A}(\Lambda_+)$ such that $\Phi^{\Z_2}_L(\rho_+)=\rho_-$. Let $L_-^i$, $i\in\N$, be an infinite family of distinct Lagrangian fillings of $\La_-$ such that the collection of Newton polytopes $N(\e_{L_-^i},\rho_-)$ is infinite, up to unimodular equivalence, and they can be distinguished by lattice point counts. This family of fillings exists by hypothesis and we can (and do) assume, after possibly considering a subcollection, that the count of lattice points strictly increases with $i\in\N$. Let us denote $b=b_1(L_-^i)$ for any $i\in\N$.

We claim that the collection of Newton polytopes $N(\e_{L\cup L_-^i},\rho_+)=N(\e_{L_-^i}\circ \Phi^{\Z_2}_L,\rho_+)$ is also infinite, up to unimodular equivalence. This will conclude the proof. Indeed, the explicit description of $\Phi_L$ above implies that each such polytope $N(\e_{L_-^i}\circ \Phi^{\Z_2}_L,\rho_+)$ is the convex hull of lattice points in $\R^{b+1}$ of two types:

\begin{itemize}
    \item[(i)] Lattice points in $\R^{b+1}$ that lie in the hyperplane $\{s=0\}$. Their convex hull in the hyperplane is precisely the polytope $N(\e_{L_-^i},\rho_-)$ in $\R^{b}=\{s=0\}\sse\R^{b+1}$.\\

    \item[(ii)] Lattice points $\R^{b+1}$ that lie in the hyperplane $\{s=-1\}$.
\end{itemize}
By hypothesis, the number of lattice points of $N(\e_{L_-^i},\rho_-)$ is strictly increasing with $i\in\N$. Since $N(\e_{L_-^i},\rho_-)\sse N(\e_{L\cup L_-^i},\rho_+)$ is a facet, the lattice points in $N(\e_{L_-^i},\rho_-)$ contribute to the lattice point count for $N(\e_{L\cup L_-^i},\rho_+)$. Therefore, there exists an infinite subcollection of the set of Newton polytopes $\{N(\e_{L\cup L_-^i},\rho_+)\}_{i\in\N}$ whose lattice point count is strictly increasing. This proves the claim and therefore the required statement.
\end{proof}


\begin{proof}[Proof of Theorem~\ref{thm:inf}]
Theorem~\ref{thm:inf} now follows from Theorem~\ref{thm:main} and Proposition~\ref{cor:cobordism_decomp}.
\end{proof}

\begin{example}\label{ex:knot}
Consider the Legendrian knot $\tilde{\Lambda}\subset (\R^3, \xi_{std})$ whose Lagrangian projection is shown in Figure~\ref{fig:knot}. The Reeb chord in the shaded region is proper and contractible by ~\cite[Prop. 4.5]{CasalsNg}. Performing a pinch move at that Reeb chord allows us to construct an exact Lagrangian cobordism from $\Lambda(\beta_{11})$ to $\tilde{\Lambda}$. By Corollary~\ref{cor:cobordism_decomp}, $\tilde{\Lambda}$ is $\Z_2[\cH]$ aug-infinite for some lattice $\cH$.\hfill$\Box$
\end{example}

\begin{figure}
	\centering
	\begin{tikzpicture}[scale=1]
		\node[inner sep=0] at (0,0)		{\includegraphics[width=9cm]{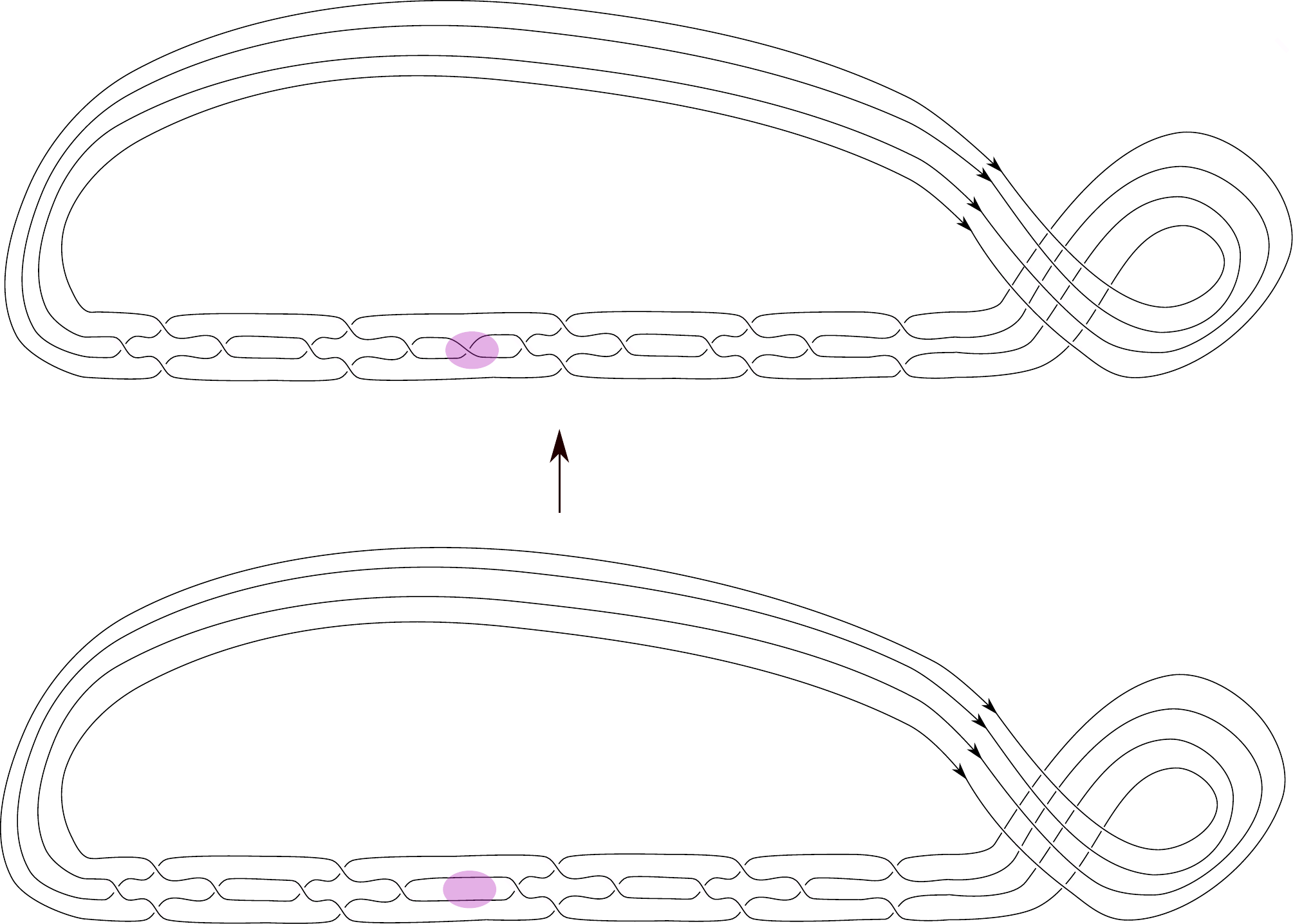}};
		\node at (-5,1.5){$\tilde{\Lambda}$};	
		\node at (-5.3,-2){$\Lambda(\beta_{11})$};
	\end{tikzpicture}
\caption{The Lagrangian projection of a Legendrian knot $\tilde{\Lambda}$ such that there exists a Lagrangian cobordism from $\Lambda(\beta_{11})$ to $\tilde{\Lambda}$ given by pinching the Reeb chord in the shaded region.}
	\label{fig:knot}
\end{figure}

\subsection{Higher-dimensional Legendrians spheres with infinitely many fillings}\label{sec:highdim}

The Legendrian links with infinitely many Lagrangian fillings discovered in \cite{Casals_gao,CasalsNg} can be front spun to produce Legendrian $S^1\times S^{n-1}\sse(\R^{2n+1}, \xi_{std})$ with infinitely many Lagrangian fillings. See \cite{ekholm_etnyre_sullivan_2005,golovko_2014} for front spinning. This is readily implied by the sheaf computations in \cite{Casals_gao,Casals_weng}, and a Floer-theoretical argument is provided in ~\cite{Golovko}. Nevertheless, the front spinning construction does not produce Legendrian spheres $S^n\sse(\R^{2n+1}, \xi_{std})$. Let us construct Legendrian spheres with infinitely many Lagrangian fillings.


\begin{proof}[Proof of Theorem~\ref{thm:highdim}]
Let $\Lambda\sse(\R^3, \xi_{std})$ be a Legendrian knot that is $\Z_2[\cH]$ aug-infinite and has no $(-1)$-degree Reeb chords. For clarity, let us also assume that the lattice $\cH$ is of the form $\cH\cong H_1(L,\Z)$ for an orientable embedded exact Lagrangian filling $L$ of $\La$. For example, to prove this theorem we choose this Legendrian knot to be $\Lambda=\tilde{\Lambda}$ from Example~\ref{ex:knot}, see Figure~\ref{fig:knot}. It is $\Z_2[\cH]$ aug-infinite and a Lagrangian filling $L$ can be taken to be a filling constructed via a pinching sequence, see Theorem \ref{thm:inf}. We fix a front projection for $\La$ to be the $(-1)$-closure of the braid word in Example ~\ref{ex:knot}.

Let $\Sigma_{S^{n-1}}\Lambda\sse(\R^{2n+1},\xi_{std})$ be its Legendrian $S^{n-1}$-spun, where we are spinning the front projection of $\Lambda$ about by an axis placed to the right of the front projection of $\Lambda$. By construction, $\Sigma_{S^{n-1}}\Lambda\sse(\R^{2n+1},\xi_{std})$ admits a Lagrangian filling $\Sigma_{S^{n-1}}L$ in the symplectization of $(\R^{2n+1},\xi_{std})$ obtained by spinning the Lagrangian filling $L$ of $\La$. Note that there is a natural inclusion map $\iota:H_1(L,\Z)\lr H_1(\Sigma_{S^{n-1}} L,\Z)$. The same argument for \cite[Theorem 1.1]{Golovko} now with $\Z_2[H_1(\Sigma_{S^{n-1}} L;\Z)]$ coefficients, as opposed to $\Z$ coefficients,
allows us to conclude that $\Sigma_{S^m}\Lambda$ is $\Z_2[H_1(\Sigma_{S^{n-1}} L);\Z]$ aug-infinite because $\tilde{\Lambda}$ is $\Z_2[\cH]$ aug-infinite. 

Let $p\in\La$ be the right cusp of the front projection of $\Lambda$ with the highest $z$ coordinate. Consider the $(n-1)$-dimensional isotropic sphere $S:=S^{n-1}\times\{p\}\sse\Sigma_{S^{n-1}}\Lambda$. Note that $S$ is the boundary of the $n$-dimensional Legendrian disk $D_S\sse(\R^{2n+1},\xi_{std})$ whose front is exactly flat and with boundary (the front of) $S$. The interior of $D_S$ is disjoint from $\Sigma_{S^{n-1}}\Lambda$ by construction. Set $m=n-1$ for ease of notation.

Performing Legendrian surgery on $\Sigma_{S^{m}}\Lambda$ along $D_S$ leads to a surgered Legendrian $\Sigma_{S^m}\Lambda_S$, where the Legendrian submanifold $\Sigma_{S^m}\Lambda_S$ is now diffeomorphic to $S^{n}$. See \cite{ArnoldSing}, or e.g.~\cite{dgr}, for Legendrian surgeries. By construction, there exists an exact Lagrangian cobordism $W_S$ from $\Sigma_{S^m}\Lambda$ to $\Sigma_{S^m}\Lambda_S$ in the symplectization of $(\R^{2n+1},\xi_{std})$. Finally, by \cite[Theorem 1.1]{dgr} and its proof, the induced dga map 
 $$\Phi_{W_S}:(\mathcal{A}(\Sigma_{S^m}\Lambda_S; \Z_2), \partial_{\Sigma_{S^m}\Lambda_S})\rightarrow (\mathcal{A}(\Sigma_{S^m}\Lambda;\Z_2); \partial_{\Sigma_{S^m}\Lambda})$$ 
 is surjective. Now, there is a Legendrian representative for $\Sigma_{S^m}\Lambda_S$ such that its set of Reeb chords is given by the set of Reeb chords of (a representative of) $\Sigma_{S^m}\Lambda$ and a Reeb chord $c_S$ with grading $|c_S|=n-k-1=0$. Since $\Lambda$ has no $(-1)$-degree Reeb chords, then $\Sigma_{S^m}\Lambda$ has no $(-1)$ degree Reeb chords, and then neither does $\Sigma_{S^m}\Lambda_S$. An adaptation of the proof of Proposition~\ref{cor:cobordism_decomp} to higher dimensions will now conclude that $\Sigma_{S^m}\Lambda_S$ is $\Z_2[H_1(\Sigma_{S^m}L\circ W_S);\Z]$ aug-infinite, as follows.

The choice of $\La=\tilde\La$ above gives a Reeb chord $\rho$ of $\La$ such that:

\begin{enumerate}
    \item The Reeb chord $\rho$ is Newton infinite with respect to a collection of augmentations $\{\varepsilon_L^i\}_{i\in\N}$,  associated to fillings $L^i$ of $\La$, and the Newton polytopes in the collection $N(\varepsilon_L^i,\rho)$, $i\in\N$, can be pairwise distinguished by their lattice point counts.\\

    \item The above item also holds after spinning, i.e.~ there exists a degree zero Reeb chord $\tilde\rho$ of $\Sigma_{S^m}\La$, which corresponds to $\rho$ after an appropriate Morsification of the front of $\Sigma_{S^m}\La$, such that $\tilde\rho$ is Newton infinite with respect to a collection of augmentations $\{\varepsilon_{\Sigma_{S^m}L}^i\}_{i\in\N}$ and the Newton polytopes in the collection $N(\varepsilon_{\Sigma_{S^m}}^i,\rho)$ are pairwise distinguished by their lattice point counts.\\
\end{enumerate}
Item (2) above holds because of the inclusion $\iota:H_1(L^i,\Z)\lr H_1(\Sigma_{S^{m}} L^i,\Z)$ and the specifics of front spinning, see \cite[Lemma 3.3]{Golovko}. The particular surgery cobordism $W_S$ above occurs near a cusp of $\Sigma_{S^m}\La$ and therefore $\tilde\rho$ naturally gives a degree zero Reeb chord $\tilde\rho_S$ in $\Sigma_{S^m}\La_S$. The augmented valued of $\tilde\rho_S$ with respect to $\Sigma_{S^m}L^i\circ W_S$ is the same augmented value as that of $\tilde\rho$ with respect to $\Sigma_{S^m}L^i$, and thus essentially the same as the augmented value of $\rho$ with respect to $\varepsilon_L^i$. For $n\geq3$, the inclusions $\iota:H_1(L^i,\Z)\lr H_1(\Sigma_{S^{m}} L^i,\Z)$ and $\iota_S:H_1(\Sigma_{S^{m}} L^i,\Z)\lr H_1(\Sigma_{S^{m}} L^i\circ W_S,\Z)$ induced by front spinning and the surgery cobordism are isomorphisms. It follows that the Newton polytopes in the collection $N(\Sigma_{S^m}L^i\circ W_S,\tilde\rho_S)$, $i\in\N$, can be pairwise distinguished via lattice point counts, as required. In the case of $n=1$, both maps $\iota$ and $\iota_S\circ\iota$ are injective, which still suffices in order to reach the same conclusion. Indeed, $\iota$ is injective by the front spinning construction (e.g.~by the K\"unneth formula) and $\iota_S$ has a rank 1 kernel given by the complement $im(\iota)^\perp$ of the image $im(\iota)$ of $\iota$ inside $H_1(\Sigma_{S^{m}} L^i,\Z)$; the latter holds because the surgery is performed along the disk $D_S$ which precisely bounds a generator for this class generating the kernel. In conclusion, for all $n\geq2$, the Legendrian $n$-sphere $\Sigma_{S^m}\Lambda_S$ is $\Z_2[H_1(\Sigma_{S^m}L\circ W_S);\Z]$ aug-infinite, which proves the statement.
\end{proof}

 \begin{figure}
	\centering
	\begin{tikzpicture}[scale=1.5]
		\node[inner sep=0] at (0,0) {\includegraphics[width=7 cm]{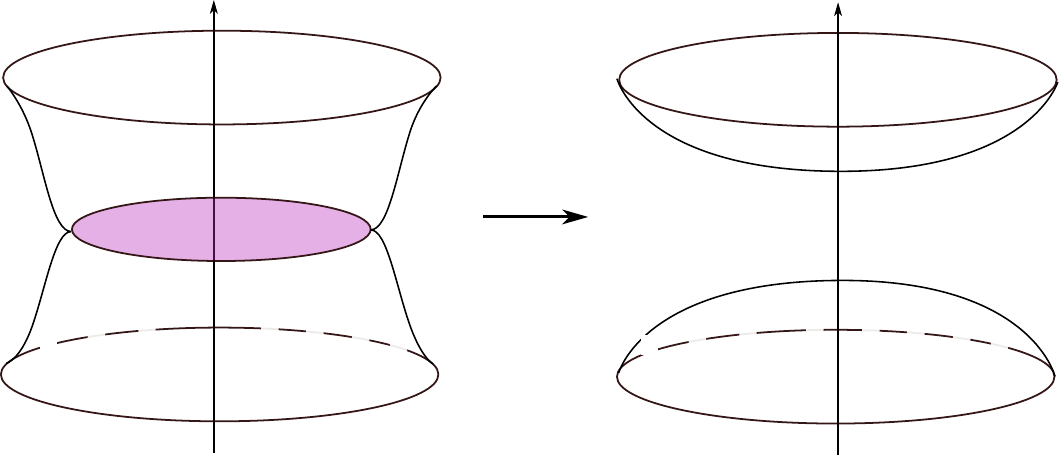}};
	\end{tikzpicture}
	\caption{An example of performing Legendrian surgery on a disk (in purple) with boundary on a spun Legendrian surface, where the vertical arrow indicates the axis of rotation.}
	\label{fig:high_pinch}
\end{figure}

\section{Two Final Remarks}\label{sec:rem}

In this last section we present two brief results which hopefully can help start expanding the study of Newton polytopes in Floer theory. First, we show in Subsection \ref{ssec:Simplex} that for any degree 0 Reeb chord of the max-tb torus knots $\La(2,n)$, expressed as the rainbow closure of $\sigma_1^n$, there exists a Lagrangian filling whose augmentation sends that Reeb chord to a Laurent polynomial whose Newton polytope is a top-dimensional simplex. Second, in Subsection \ref{sec:nonorientable} we distinguish non-orientable Lagrangian fillings and prove Theorem \ref{thm:nonorientable}.

\subsection{Top-dimensional simplices as Newton polytopes for augmentations}\label{ssec:Simplex}


It would be interesting to explore fully what type of lattice polytopes can appear as $N(\e_L,\rho)$, where $L$ is a Lagrangian filling of a Legendrian knot $\La$ and $\rho$ a Reeb chord of $\La$. To start, we show that top-dimensional simplices always appear for one of the simplest types of Legendrian knots, in the following sense:

\begin{prop}\label{prop:torusknots}
Let $\Lambda(2,n)$ be the max-tb representative of the $(2,n)$-torus knot, $n$ odd, presented as the rainbow closure of the positive braid $\sigma_1^n$. For every zero graded Reeb chord $b$ of $\Lambda(2,n)$, there exists a Lagrangian filling $L$ with induced augmentation $\e_L$ such that the Newton polytope $N(\e_L,b)\sse\R^{n-1}$ is $($equivalent to$)$ the standard top-dimensional simplex with volume $\frac{1}{(n-1)!}$.
\end{prop}

\begin{proof}
Let us draw $\La(2,n)$ in the Lagrangian projection, as the Ng resolution of the rainbow closure of the positive braid $\sigma_1^n$. The degree 0 Reeb chords correspond to the crossings of $\sigma_1^n$. Suppose that $b$ is the $i$-th crossing, starting the labeling from the left. To construct $L$, we consider the following pinching sequence:

\begin{itemize}
    \item[(1)] First pinch the crossings to the right of $b$ in left-to-right order.

    \item[(2)] Then, pinch the crossings to the left of $b$, in right-to-left order
    
    \item[(3)] Finally pinch the crossing $b$ itself.
\end{itemize}
For example, consider $\Lambda(2,7)$ presented as the rainbow closure of $\sigma_1^7$. Suppose that $b=b_5$ is the 5th crossing (degree 0 Reeb chord) of $\sigma_1^7$ counting from the left. Then the pinching sequence is given by $b_6,b_7,b_4,b_3,b_2,b_1,b_5$. Indeed, the crossings to the right of $b=b_5$ are $b_6,b_7$ ordered left-to-right, so we first pinch $b_6$ and then $b_7$. Similarly, the crossings to the left of $b=b_5$ are $b_1,b_2,b_3,b_4$ ordered left-to-right: we first pinch starting from the rightmost crossing $b_4$ and then $b_3$, then $b_2$ and finally $b_1$, going right-to-left. The last step is pinching $b=b_5$ itself.

Note that it does not matter in what order we intersperse the opening of the crossings to the right of $b$ and to the left of $b$; we chose first to pinch those on the right just for specificity. After these $n$ pinch moves, we have two max-tb unknots that we can fill with two minimum cobordisms each to obtain the Lagrangian filling $L$, and thus an augmentation $\e_L$. The $\e_L$-augmented value of $b$ is
$$\e_L(b)=s_i+\sum_{j=1}^{i-1}s_{j}^{-1}\cdot\left(\prod_{k=j+1}^{i-1} s_{k}^{-2}\right)+\sum_{j=i+1}^{n}s_{j}^{-1}\cdot\left(\prod_{k=i+1}^{j-1} s_{k}^{-2}\right).$$
This is readily computed via the same methods as in \cite{ehk,Pan}.
By construction, the relation $\e_L(s_1s_2\ldots s_n)=1$ holds, and therefore we set $s_i=s^{-1}_1\cdot\ldots \cdot s_{i-1}^{-1}\cdot s_{i+1}^{-1}\cdot\ldots\cdot s^{-1}_{n-1}s_n^{-1}$. The Newton polytope $N(\e_L,b)$ is then a polytope in $\Z^{n-1}\sse\R^{n-1}$, the lattice associated to $s_1,\ldots,s_{i-1},s_{i+1},\ldots,s_{n}$. It is the convex hull of the $n$ monomials in $\e_L(b)$. Thanks to the computation above, $N(\e_L,b)$ is in fact the convex hull of the following points in $\bR^{n-1}$:
\begin{itemize}
\item[(a)] (Type-R) The lattice points associated to the monomials coming from pinchings to the right of $b$, which is the third summand in $\e_L(b)$:
$$(0,\ldots,0,-1_{i+1},0,\ldots,0), (0,\ldots,0,-2_{i+1},-1_{i+2},0,\ldots,0), (0,\ldots,0,-2_{i+1},-2_{i+2},-1_{i+3},\ldots,0),$$ $$\ldots, (0,\ldots,0,-2_{i+1},-2_{i+2},-2_{i+3},\ldots,-1_{n-2},0),(0,\ldots,0,-2_{i+1},-2_{i+2},-2_{i+3},\ldots,-1_{n-1}),$$
where the subindices indicate the component of the point.\\

\item[(b)] (Type-L) The lattice points associated to the monomials coming from pinchings to the left of $b$, which is the second summand in $\e_L(b)$:
$$(0,\ldots,0,-1_{i-1},0,\ldots,0), (0,\ldots,0,-1_{i-2},-2_{i-1},0,\ldots,0), (0,\ldots,0,-1_{i-3},-2,-2_{i-1},\ldots,0),$$
$$\ldots, (0,-1_{2},-2_{3},\ldots,-2_{i-1},0,\ldots,0),(-1_{1},-2_{2},\ldots,-2_{i-1},0,\ldots,0).$$

\item[(c)] The lattice point $(-1,-1,\ldots,-1)$ associated to the monomial $s_i=s^{-1}_1\cdot\ldots \cdot s_{i-1}^{-1}\cdot s_{i+1}^{-1}\cdot\ldots\cdot s^{-1}_{n-1}s_n^{-1}$, which is the first summand of $\e_L(b)$.
\end{itemize}

Let us translate $N(\e_L,b)$ by adding the constant vector $(1,\ldots,1)$. The lattice point in type (c) above becomes the origin. The resulting $(n-1)$ lattice points can then be considered as vectors in $\R^{n-1}$, centered at the origin. Let $S\in M_{n-1}(\Z)$ be the matrix constructed by using these vectors as rows, for a choice of row permutation. It suffices to argue that $S\in \GL_{n-1}(\Z)$, i.e.~ $|\det(S)|=1$. Given $m,l\in\N$, consider the matrix $O_{l,m}\in M_{l,m}(\Z)$ all of whose entries are $1$, and the matrix $D_m\in M_{m}(\Z)$ which has entries equal to $0$ in the diagonal, entries equal to $1$ above the diagonal, and entries equal to $-1$ below the diagonal. Then $S$ is given in block notation as

$$S=\begin{pmatrix}
D_{i-1} & \vline & O_{i-1,n-i} \\
\hline
O_{n-i,i-1} & \vline & D_{n-i}
\end{pmatrix}\in M_{n-1}(\Z).$$

Consider matrices  $\mathbb{O}_{i-1,n-i},
\mathbb{O}_{n-i,i-1}$ such that each odd row starts on the left with $1$, each even row starts on the left with $-1$, and for each row (even or odd) the rest of the entries alternate between $1$ and $-1$, so that the entries of a row are always alternating when read left to right. Let $\mathbb{D}_{m}\in M_m(\Z)$ be a matrix with diagonal entries $0$, and such that: to the right of each (zero) diagonal entry, the entries in that piece of the row alternate between $\pm1$ starting with $-1$ to the right of the diagonal entry, and to the south of each diagonal entry, the entries in that piece of the column alternate between $\pm1$ starting with $1$ right below the diagonal entry. Then $S$ admits the explicit inverse $$S^{-1}=\begin{pmatrix}
\mathbb{D}_{i-1} & \vline & \mathbb{O}_{i-1,n-i} \\
\hline
\mathbb{O}_{n-i,i-1} & \vline & \mathbb{D}_{n-i}
\end{pmatrix}\in M_{n-1}(\Z),$$
where the equality $S\cdot S^{-1}=\mbox{Id}$ is verified directly. Therefore $|\det(S)|=1$.
\end{proof}

\subsection{Distinct non-orientable exact Lagrangian fillings}\label{sec:nonorientable}
The existence and classification of non-orientable exact Lagrangian fillings is less understood than that of their orientable counterparts. By distinguishing augmentations with $\Z_2[H_1(L, \Lambda;\Z)]$ coefficients we can distinguish non-orientable exact Lagrangian fillings. First, we must establish how to keep track of generators of relative homology for non-orientable decomposable Lagrangian cobordisms $L$.

\subsubsection{Non-orientable Lagrangian cobordisms} By definition, a non-orientable pinch move in the Lagrangian projection replaces a negative crossing corresponding to a contractible proper Reeb chord with its $0$-resolution as shown in Figure~\ref{fig:nonorientable_pinch}. If two Legendrians $\Lambda_-$ and $\Lambda_+$ are related by an non-orientable pinch move, then there exists an exact Lagrangian non-orientable cobordism from $\Lambda_-$ to $\Lambda_+$ which is topologically a non-orientable $1$-handle, i.e. a twice punctured $\R\P^2$ union a collection of trivial cylinders, all with boundary in $\Lambda_{+}\sqcup \Lambda_{-}$.

Let $\Lambda$ be a Legendrian link in $(\mathbf{R}^3,\xi_{std})$ and $L$ an exact decomposable Lagrangian filling of $\Lambda$. For such cobordisms, we give a relative basis of $H_1(L, \Lambda;\Z)$ in a standardized manner following~\cite{ehk,Pan,CasalsNg}, encoded by basepoints on the Legendrian slices of $L$. The set up in this non-orientable case is similar to \cite{CasalsNg}, as follows.

Assume that there is a collection of oriented arcs and circles on $L$ such that all circles are contained in the interior of $L$, and the oriented arcs $\{\gamma_1, \ldots, \gamma_k\}$ are transverse to  $\Lambda_+$ and $\Lambda_-$ at their endpoints on $\Lambda_+$ and $\Lambda_-$. Then the induced dga chain map $\Phi_{L}$ is defined by counting holomorphic disks with boundary on $L$ and boundary punctures mapping to Reeb chords on $\Lambda_+$ and $\Lambda_-$. One can assume that the curves intersect $\gamma_i$ transversely, and that no arc $\gamma_i$ has an endpoint on a Reeb chord of $\Lambda_+$ or $\Lambda_-$. Let $\{\gamma_{i_1}, \ldots, \gamma_{i_p}\}$ denote the subset of arcs with at least one endpoint on $\Lambda_+$ and let $\{\gamma_{j_1}, \ldots, \gamma_{j_q}\}$ denote the subset of arcs with at least one endpoint on $\Lambda_-$. We view these endpoints $\gamma_i$ as basepoints $s_i$ on $\Lambda_{\pm}$ respectively. The dga of $\Lambda_+$ equipped with the basepoints associated to $\gamma_{i_1}, \ldots, \gamma_{i_p}$ has coefficient ring $\Z_2[s^{\pm1}_{i_1}, \ldots, s^{\pm1}_{i_p}]$, while the dga of $\Lambda_-$ equipped with the basepoints associated to $\gamma_{j_1}, \ldots, \gamma_{j_q}$ has coefficient ring $\Z_2[s^{\pm1}_{j_1}, \ldots,s^{\pm1}_{j_q}]$. We tensor both dgas with $\Z_2[s_1^{\pm1}, \ldots, s^{\pm1}_k]$ in order to define $\Phi_{L}$ as a chain map of $\Z_2[s_1^{\pm1}, \ldots s_k^{\pm1}]$-algebras. If $\Delta$ is a holomorphic disk contributing to $\Phi_{L}$ then we give it the coefficient $s_1^{n_1(\Delta)}\cdots s_k^{n_k(\Delta)}\in R[s_1^{\pm1}, \ldots s_k^{\pm1}]$ where $n_i(\Delta)$ is the signed intersection number of $\partial \Delta$ with $\gamma_i$. Observe that if we have a basis of generators $\{\gamma_i\}$ for $H_1(L, \Lambda_+\cup \Lambda_-;\Z)$, then for each holomorphic (punctured) disk $\Delta$, $n_i(\Delta)$ allows us to keep track of the relative homology class of $\partial \Delta$.


Suppose that $L$ is constructed using a sequence of Legendrian isotopies, $n$ pinch moves and (at the end) $k$ minimum cobordisms. (See end of Subsection \ref{sec:prelim} above to recall these concepts, right before Remark \ref{rmk:dga_map_non_orientable}.) Let $\Lambda_0$ be an unlinked set of $k$ max-tb Legendrian unknots. Let $\Lambda_i$ denote the Legendrian link obtained from $\La_0$ after the $i$th reverse pinch move on $\Lambda_0$ and let $L_i$ denote the exact Lagrangian cobordism from $\Lambda_{i}$ to $\Lambda_{i+1}$. Place $k$ basepoints $t_1, \ldots, t_k$ on the Legendrian $\Lambda=\Lambda_+$ being filled. Each $\Lambda_i$ will have two additional basepoints than $\Lambda_{i+1}$, which we label $s_i^{\pm1}$, $\Lambda_0$ contains the $k$ basepoints $t_1, \ldots, t_k$, and $2n$ basepoints $s_i^{\pm1}$ for $i=1,\ldots, n$. Let $\tau_i$ denote the arc traced out by the basepoints $t_i$. The two basepoints $s_i$ will be the endpoints of an arc $\sigma_i$ from $\Lambda_0$ to itself. While the cobordism itself may be non-orientable, we assume that each Legendrian $\Lambda_i$ and each arc $\tau_i$ and $\sigma_i$ are oriented. The intersection of the oriented $\Lambda_i$ with the oriented $\sigma_i$ and $\tau_i$ will be encoded by $s_i$ or $s_i^{-1}$, and $t_i$ or $t_i^{-1}$ depending on whether the intersection is positive or negative. See Figure~\ref{fig:nonorientable_pinch} for an example of this. We always choose orientations on $s_i$ and $\Lambda_i$ so that the dga maps induced by pinch moves are well-defined.

Let $\Lambda_{-}$ be obtained from $\Lambda_{+}$ by a non-orientable pinch move. If both $\Lambda_{+}$ and $\Lambda_{-}$ are knots, we will set the convention of choosing the orientation on $\Lambda_{-}$ so that the arc $\tau$ between the basepoints intersects both $\Lambda_{-}$ and $\Lambda_{+}$ positively. If $\Lambda_{+}$ is a link and $\Lambda_{-}$ is a link with one less link component, then we require that the intersection $\tau_i\cap \Lambda_{+}$ has the opposite intersection sign of $\tau_i\cap \Lambda_{-}$, for some $i$.

For any orientable or non-orientable exact filling $L$ of $\Lambda$, $L$ is a surface with boundary and thus deformation retract onto its $1$-skeleton. In particular, $H_1(L, \Lambda;\Z)$ has no torsion, $\Z_2[H_1(L,\Lambda;\Z)$ is a Laurent polynomial ring and $H_1(L,\Lambda;\Z)\simeq \langle \gamma_1, \ldots, \gamma_k\rangle$ is isomorphic to a free lattice of rank $k$ for some $k\in\N$. All of the results in Sections~\ref{sec:prelim} and~\ref{sec:newtonbackground} stated for $H_1(L;\Z)$ also hold for $H_1(L, \partial L; \Z)$.

\begin{remark}
Suppose that we have constructed a non-orientable filling of $\Lambda$ with $n$ pinch moves. The square of the product of a subset of the basepoints corresponding to the $n$ pinch moves will be equal to $1$ or to the basepoint $t$. This is a consequence of the fact that these pinch moves give surfaces that are topologically $\R\mathbb{P}^2$s with punctures.\hfill$\Box$
\end{remark}

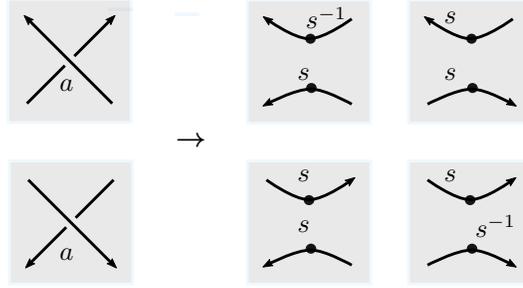
\begin{figure}
	\centering
	\begin{tikzpicture}[scale=1.5]
		\node[inner sep=0] at (0,0) {\includegraphics[width=7 cm]{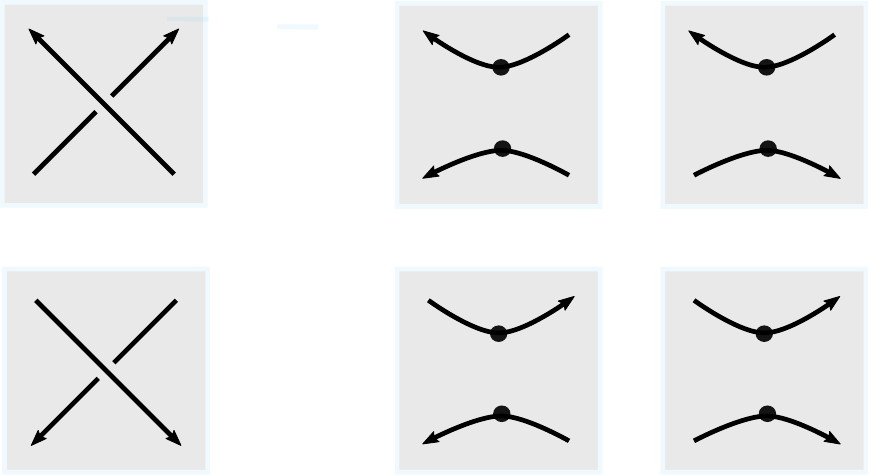}};
		\node at (-1.8,0.5){$a$};
		\node at (-1.8,-1){$a$};
		\node at (.3,0.6){$s$};
		\node at (.5,1.1){$s^{-1}$};
		\node at (1.6,1.1){$s$};
		\node at (1.6,0.6){$s$};
		\node at (.3,-0.75){$s$};
		\node at (.3,-.3){$s$};
		\node at (1.6,-.3){$s$};
		\node at (2,-0.75){$s^{-1}$};
		\node at (-0.7,0){$\boldsymbol{\rightarrow}$};
	\end{tikzpicture}
	\caption{Non-orientable pinch moves at a Reeb chord $a$, where the crossing $a$ is replaced with its $0$-resolution and two basepoints $s, s^{-1}$. After performing either of the non-orientable pinch moves shown on the left, the arcs may be oriented in either of the four ways shown on the right. The basepoints are determined by the orientation of the arcs.} 
	\label{fig:nonorientable_pinch}
\end{figure}

\subsubsection{Examples of non-orientable fillings} Legendrian knots with non-orientable Lagrangian fillings can be readily constructred as follows. Start with a Legendrian knot $\La_-$ which admits an orientable exact embedded Lagrangian filling $L_-$ and attach a non-orientable Lagrangian 1-handle, i.e. perform the inverse of a pinch move in Figure \ref{fig:nonorientable_pinch}. This yields a Legendrian link $\La_+$ and a non-orientable exact Lagrangian cobordism $N$ from $\La_-$ to $\La_+$. Then $\La_+$ admits a non-orientable Lagrangian filling $L_+=L\cup_{\La_-}N$ constructed by concatenating $L$ with $N$. Here are two simple instances:

\begin{itemize}
    \item[(i)] Let $\La_+=\Lambda(2,-2n)$ be a max-tb negative Legendrian torus link, $n\geq 2$ whose projection is given by taking the positive Legendrian torus link $\Lambda(2,2n)$ given as the rainbow closure of the positive braid $\sigma_1^{2n}$ and reversing the orientation on one of the link components. After a non-orientable pinch move, one can choose orientations so that $\La_-$ is the max-tb Legendrian torus knot $\La(2,2n-1)$. More generally, it follows from \cite{Pan} that one can produce at least a Catalan number $C_{2n}$ of distinct non-orientable Lagrangian fillings for $\La_+$ by using pinching sequences to create (non-orientable) fillings $L$ of $\La_+$. Indeed, the systems of augmentations over $\Z[H_1(L, \Lambda;\Z)]$ distinguish these fillings because the computations in \cite{Pan}, counting number of monomials, apply.\\

    \item[(ii)] Let $\La_+=\Lambda(K_n)$ be a hyperbolic link whose Lagrangian projection is shown in Figure~\ref{fig:Kn}, $n\geq2$. If $n$ is even, we can performing a non-orientable pinch move at the Reeb chord $b_2$. This yields $\La_-=\La(2,n-1)$. If $n$ is odd, perform an orientable pinch move at $b_2$, so as to obtain $\La_-=\La(2,-n-1)$. In either case, concatenating the Lagrangian fillings for $\La_-$ with the pinch moves yields non-orientable Lagrangian fillings for $\La(K_n)$.
\end{itemize}

\begin{figure}
	\centering
	\begin{tikzpicture}[scale=1]
		\node[inner sep=0] at (0,0) {\includegraphics[width=5cm]{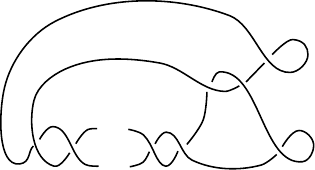}};
		\node at (-1.9,-1.5){$c_1$};
		\node at (-1.3,-1.5){$c_2$};
		\node at (-0.5,-1){$\cdots$};
		\node at (0,-1.5){$c_{n-1}$};
		\node at (0.6,-1.5){$c_n$};
		\node at (0.8,0.4){$b_1$};
		\node at (1.45,0.4){$b_2$};
		\node at (2,-1.5){$a_2$};
		\node at (1.9,0.9){$a_1$};
	\end{tikzpicture}
\caption{The Lagrangian projection of a twist knot $K_n$.}
	\label{fig:Kn}
\end{figure}
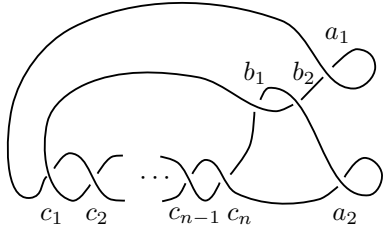

In the examples above, only finitely many non-orientable Lagrangian fillings can be distinguished. In line with \cite{Casals_Viterbo60}, one would naturally conjecture that only finitely many such actually exist. In order to prove Theorem \ref{thm:nonorientable} we need to construct a different example. We first present a Legendrian link with infinitely many non-orientable Lagrangian fillings, as follows.

\begin{example}\label{ex:nonorientable}
Let $\tilde{\Lambda}(\beta_{11})$, resp.~$\tilde{\Lambda}_1$, be the oriented Legendrian links given by the $(-1)$-closure of the positive braid $\beta_{11}$, resp.~$\Lambda_1$, with only one of the link components having the opposite orientation. Construct a filling using the same pinching sequences as in Subsection \ref{sec:proofmain}, except using a non-orientable pinch move when required. The switch in orientation of one component will require such non-orientable pinch move, and there the resulting fillings are non-orientable. By construction, the induced augmentations and monodromy maps of the isotopy loops in Theorem~\ref{thm:main} are unchanged over characteristic 2. Since all of the $0$ graded Reeb chords $a_i$ of $\Lambda(\beta_{11})$ and $\Lambda_1$ are cycles, then the corresponding Reeb chords $a_i$ of $\tilde{\Lambda}(\beta_{11})$ and $\tilde{\Lambda}_1$ are also cycles. Thus, the computation in Subsection \ref{sec:proofmain} concludes that there are augmentations $\e_L^i$ of $\tilde{\Lambda}(\beta_{11})$ and $\tilde{\Lambda}_1$ such that $N(\e_L^i, a_9)$ are not pairwise unimodular equivalent. Hence, the non-orientable exact Lagrangian fillings $L^i$ are not Hamiltonian isotopic. The fillings of $\tilde{\Lambda}(\beta_{11})$ construced in this manner are of genus $1$ and crosscap number 1, while the fillings of $\tilde{\Lambda}_1$ are of genus 2 and crosscap number 1.\hfill$\Box$
\end{example}

The only difference between Example \ref{ex:nonorientable} and Theorem \ref{thm:nonorientable} is that the former yields a link, rather than a knot. This is readily corrected, as follows.

\subsubsection{Proof of Theorem~\ref{thm:nonorientable}} Following the proof of Proposition~\ref{cor:cobordism_decomp} it suffices to produce a Legendrian knot $\Lambda$ such that there exists a non-orientable decomposable exact Lagrangian cobordism $L$ from a Legendrian $\Lambda_-$ to $\Lambda$ where $\Lambda_-$ has an element $a\in(\mathcal{A}(\Lambda_-;\Z_2), \partial_{\Lambda_-})$ that is Newton infinite, with an infinity of those polytopes distinguished by lattice point counts, and $\Phi_L(x)=a$ for a cycle $x\in(\mathcal{A}(\Lambda;\Z_2), \partial_{\Lambda})$. An instance of such cobordism can be constructed as follows. Let $\Lambda_-=\Lambda'$ be the Legendrian knot shown in Figure~\ref{fig:knotnonorientable}. If we perform an orientable pinch move at the Reeb chord $b_2$ (marked in a pink dot), we obtain the Legendrian $\tilde{\Lambda}(\beta_{11})$ from Example \ref{ex:nonorientable}. This exact Lagrangian cobordism $L$ between $\Lambda'$ and $\tilde{\Lambda}(\beta_{11})$ induces a dga map $$\Phi_L:(\mathcal{A}(\Lambda';\Z_2[H_1(L, \Lambda'\cup \tilde{\Lambda}(\beta_{11});\Z)]), \partial_{\Lambda'})\rightarrow (\mathcal{A}(\tilde{\Lambda}(\beta_{11});\Z_2[H_1(L,\Lambda'\cup \tilde{\Lambda}(\beta_{11});\Z)]), \partial_{\tilde{\Lambda}(\beta_{11})})$$ such that $\Phi_L(b_2)=s$ and $\Phi_L(a_9)=a_9$. Note that $a_9$ is a cycle in both $\tilde{\Lambda}$ and $\tilde{\Lambda}(\beta_{11})$ and, as argued above, the Reeb chord $a_9$ of $\tilde{\Lambda}(\beta_{11})$ is Newton infinite and distinguished by lattice point counts. Thus $\Lambda'$ is $\Z_2[\cH]$ aug-infinite, and therefore has infinitely many distinct non-orientable exact Lagrangian fillings. (Here the fillings have genus 1 and crosscap number 2.) This concludes the result. \hfill$\Box$

\begin{figure}
	\centering
	\begin{tikzpicture}[scale=1]
		\node[inner sep=0] at (0,0)		{\includegraphics[width=9cm]{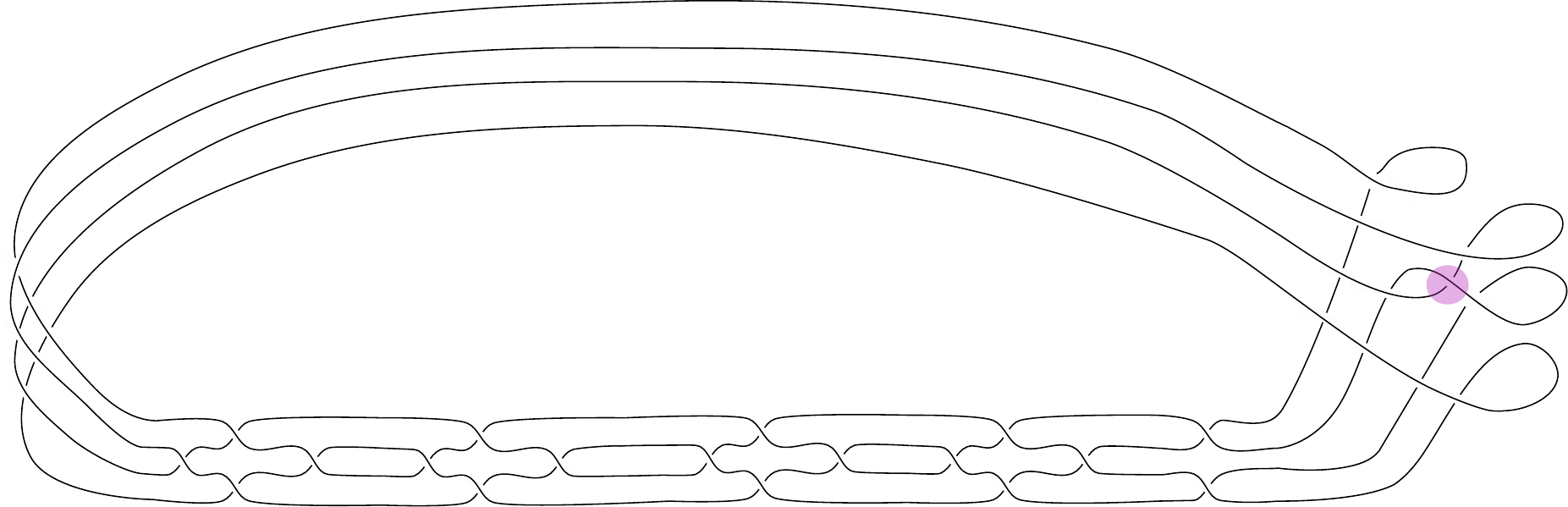}};
	\end{tikzpicture}
\caption{The Lagrangian projection of a Legendrian knot $\Lambda'$ such that there exists a Lagrangian cobordism from $\tilde{\Lambda}(\beta_{11})$ to $\Lambda'$ given by pinching the Reeb chord in the shaded region which we label by $b_2$. This then implies that $\Lambda'$ has infinitely many distinct non-orientable exact Lagrangian fillings.}
	\label{fig:knotnonorientable}
\end{figure}


\bibliographystyle{alpha}
\bibliography{InfiniteFillingsChar2_references}

\end{document}